\documentclass{amsart}

\usepackage{amssymb,mathtools,comment,hyperref}
\usepackage[shortlabels]{enumitem}
\usepackage[alphabetic,initials,nobysame]{amsrefs}
\BibSpec{article}{%
    +{}  {\PrintAuthors}                {author}
    +{,} { \textit}                     {title}
    +{.} { }                            {part}
    +{:} { \textit}                     {subtitle}
    +{,} { \PrintContributions}         {contribution}
    +{.} { \PrintPartials}              {partial}
    +{,} { }                            {journal}
    +{}  { \textbf}                     {volume}
    +{}  { \PrintDatePV}                {date}
    +{,} { \issuetext}                  {number}
    +{,} { \eprintpages}                {pages}
    +{,} { }                            {status}
    +{,} { \url}                        {url}
    +{,} { \PrintDOI}                   {doi}
    +{,} { available at \eprint}        {eprint}
    +{}  { \parenthesize}               {language}
    +{}  { \PrintTranslation}           {translation}
    +{;} { \PrintReprint}               {reprint}
    +{.} { }                            {note}
    +{.} {}                             {transition}
    +{}  {\SentenceSpace \PrintReviews} {review}
}

\BibSpec{collection.article}{%
	+{}  {\PrintAuthors}				{author}
	+{,} { \textit}                     {title}
	+{.} { }                            {part}
	+{:} { \textit}                     {subtitle}
	+{,} { \PrintContributions}         {contribution}
	+{,} { \PrintConference}            {conference}
	+{}  {\PrintBook}                   {book}
	+{,} { }                            {booktitle}
	+{,} { }							{series}
	+{,} { }							{publisher}
	+{,} { }							{address}
	+{,} { \PrintDateB}                 {date}
	+{,} { pp.~}                        {pages}
	+{,} { }                            {status}
	+{,} { \PrintDOI}                   {doi}
	+{,} { available at \eprint}        {eprint}
	+{}  { \parenthesize}               {language}
	+{}  { \PrintTranslation}           {translation}
	+{;} { \PrintReprint}               {reprint}
	+{.} { }                            {note}
	+{.} {}                             {transition}
	+{}  {\SentenceSpace \PrintReviews} {review}
}
\theoremstyle{plain}
\newtheorem{theorem}{Theorem}
\newtheorem{lemma}[theorem]{Lemma}

\newtheorem{prop}[theorem]{Proposition}
\newtheorem{corollary}[theorem]{Corollary}

\theoremstyle{definition}
\newtheorem{definition}[theorem]{Definition}

\theoremstyle{remark}
\newtheorem{remark}[theorem]{Remark}
\newtheorem{question}[theorem]{Question}
\newtheorem{example}[theorem]{Example}

\numberwithin{equation}{section}
\numberwithin{theorem}{section}
\numberwithin{conjecture}{section}
\newcommand{\x}{\scalebox{1.2}{$\chi$} } 

\newcommand{\br}{\overline}
\newcommand{\R}{\mathbb R}
\newcommand{\C}{\mathbb C}

\newcommand{\Z}{\mathbb Z}
\newcommand{\N}{\mathbb N}

\DeclareMathOperator{\dist}{\textup{\text{dist}}}
\DeclareMathOperator{\diam}{\textup{\text{diam}}}

\DeclareMathOperator{\Area}{\textup{\text{Area}}}

\DeclareMathOperator{\loc}{\textup{loc}}

\begin{document}
\title[Monotone Sobolev functions]{Monotone Sobolev functions in planar domains: level sets and smooth approximation}
\author{Dimitrios Ntalampekos}
\address{Institute for Mathematical Sciences, Stony Brook University, Stony Brook, NY 11794, USA.}
\email{dimitrios.ntalampekos@stonybrook.edu}

\date{\today}
\keywords{Sobolev, monotone, level set, approximation, Sard's theorem}
\subjclass[2010]{Primary 30E10; Secondary 35J92, 41A65, 46E35.}

\begin{abstract}
We prove that almost every level set of a Sobolev function in a planar domain
consists of points, Jordan curves, or homeomorphic copies of an interval. For
monotone Sobolev functions in the plane we have the stronger conclusion that
almost every level set is an embedded $1$-dimensional topological submanifold of the plane. Here
monotonicity is in the sense of Lebesgue: the maximum and minimum of the
function in an open set are attained at the boundary. Our result is an analog
of Sard's theorem, which asserts that for a $C^2$-smooth function in a planar domain
almost every value is a regular value. As an application we show that monotone
Sobolev functions in planar domains can be approximated uniformly and in the
Sobolev norm by smooth monotone functions. 
\end{abstract}

\maketitle

\section{Introduction}\label{Section:Introduction}

The classical theorem of Sard \cite{Sard:Theorem} asserts that for a $C^2$-smooth function $f$ in a planar domain $\Omega$ almost every value is a regular value. That is, for almost all $t\in \R$ the set $f^{-1}(t)$ does not intersect the critical set of $f$, and hence, $f^{-1}(t)$ is an embedded $1$-dimensional $C^2$-smooth submanifold of the plane. This theorem is sharp, in the sense that the $C^2$-smoothness cannot be relaxed to $C^1$-smoothness, as was shown by Whitney \cite{Whitney:SardOptimal}. 

In fact, Sard's theorem and some of the other theorems that we quote below have more general statements that hold for maps defined in subsets of $\R^n$, taking values in $\R^m$, and having appropriate regularity. In order to facilitate the comparison to our results, we will only give formulations in the case of real-valued functions defined in planar domains.

Several generalizations and improvements of Sard's theorem have been proved since the original theorem was published. In particular, Dubovitskii \cite{Dubovitskii:Sard} proved that a $C^1$-smooth function $f$ in a planar domain has the property that for a.e.\ value $t\in \R$ the set $f^{-1}(t)$ intersects the critical set in a set of Hausdorff $1$-measure zero. De Pascale \cite{DePascale:Sard} extended the conclusion of Sard's theorem to Sobolev functions of the class $W^{2,p}$, where $p>2$. For other versions of Sard's theorem in the setting of H\"older and Sobolev spaces see  \cite{Bates:Sard}, \cite{BojarskiHajlaszStrzelecki:Sard}, \cite{Figalli:Sard}, \cite{Norton:Sard}. 

Now, we turn our attention to the structure of the level sets of functions, instead of discussing the critical set. Theorem 1.6 in \cite{BojarskiHajlaszStrzelecki:Sard} states that if $f\in W^{1,p}_{\loc}(\R^2)$, then there exists a Borel representative of $f$ such that for a.e.\ $t\in \R$ the level set $f^{-1}(t)$ is equal to $Z\cup \bigcup_{j\in \N} K_j$, where $\mathcal H^1(Z)=0$, $K_j\subset K_{j+1}$, and $K_j$ is \textit{contained} in an $1$-dimensional $C^1$-smooth submanifold $S_j$ of $\R^2$ for $j\in \N$.

Under increased regularity, Bourgain, Korobkov, and Kristensen \cite{BourgainEtAl:Sard} proved that if $f\in W^{2,1}(\R^2)$ then for a.e.\ $t\in \R$ the level set $f^{-1}(t)$ is an $1$-dimensional $C^1$-smooth manifold. We direct the reader to \cite{BourgainEtAl:Sard} and the references therein for a more detailed exposition of generalizations of Sard's theorem.

Our first theorem is the following:

\begin{theorem}\label{Intro:Sobolev:Theorem}
Let $\Omega\subset \R^2$ be an open set and $u\colon \Omega \to \R$ be a continuous function that lies in $W^{1,p}_{\loc}(\Omega)$ for some $1\leq p\leq \infty$. Then for a.e.\ $t\in \R$ the level set $u^{-1}(t)$ has locally finite Hausdorff $1$-measure and each component of $u^{-1}(t)$ is either a point, or a Jordan curve, or it is homeomorphic to an interval.
\end{theorem}

This result generalizes a result of Alberti, Bianchini, and Crippa \cite[Theorem 2.5(iv)]{AlbertiEtAl:LipschitzLevelSets}, who obtained the same conclusion under the weaker assumptions that $u$ is Lipschitz and compactly supported.

Under no further regularity assumptions, we do not expect in this case the level sets to be $1$-dimensional manifolds. Instead, we pose some topological restriction:

\begin{definition}\label{Monotone:Definition}
Let $\Omega\subset \R^n$ be an open set and $f\colon \Omega\to \R$ be a continuous function. We say that $f$ is monotone if for each open set $U\subset \subset \Omega$ the maximum and the minimum of $f$ on $\br U$ are attained on $\partial U$. That is,
\begin{align*}
\max_{ \br U}f= \max_{\partial U}f \,\,\, \textrm{and} \,\,\, \min_{\br U}f=\min_{\partial U}f.
\end{align*}  
\end{definition}
Here the notation $U\subset \subset \Omega$ means that $\br U$ is compact and is contained in $\Omega$. This definition can also be extended to real-valued functions defined in a locally compact metric space $X$. 
\begin{remark}\label{Monotone:Remark}
If $f$ extends continuously to $\br \Omega$, then we may take $U\subset\subset \br \Omega$ in the above definition.
\end{remark}

Monotonicity in Definition \ref{Monotone:Definition} is in the sense of Lebesgue \cite{Lebesgue:Monotone}. There are other more general notions of monotonicity; e.g.\ there is a notion of \textit{weak monotonicity} due to Manfredi \cite{Manfredi:Monotone} that agrees with Lebesgue monotonicity for the spaces $W^{1,p}$, $p\geq 2$. We now state our main theorem:

\begin{theorem}\label{Intro:MonotoneSobolev:Theorem}
Let $\Omega\subset \R^2$ be an open set and $u\colon \Omega \to \R$ be a continuous monotone function that lies in $W^{1,p}_{\loc}(\Omega)$ for some $1\leq p\leq \infty$. Then for a.e.\ $t\in \R$ the level set $u^{-1}(t)$ has locally finite Hausdorff $1$-measure and it is an embedded $1$-dimensional topological submanifold of $\R^2$.
\end{theorem}
Recall that a set $J\subset \R^2$ is an embedded $1$-dimensional sumbanifold of $\R^2$ if for each point $x\in J$ there exists an open set $U$ in $\R^2$ and a homeomorphism $\phi\colon U\to \R^2$ such that $\phi(J)=\R$. By the classification of $1$-manifolds (see Theorem \ref{Classification:theorem}), each component of $J$ is homeomorphic to $S^1$ or to $\R$.

Monotone Sobolev functions appear as infimizers of certain energy functionals, not only in the Euclidean setting, such as $\mathcal A$-harmonic functions \cite{Heinonenetal:DegenerateElliptic}, but also in metric spaces. For example, they appear as real and imaginary parts of ``uniformizing" quasiconformal mappings into the plane from metric spaces $X$ that are homeomorphic to the plane; see \cite{Rajala:uniformization}, \cite{RajalaRomney:reciprocal}. Our proof is partly topological and can be applied also in these settings, leading to the understanding of the level sets of the ``uniformizing" map, which is crucial for proving injectivity properties.  The results in this paper can be used to simplify some of the arguments in \cite{Rajala:uniformization}. More specifically, our techniques yield the following result in the metric space setting.

\begin{theorem}\label{Intro:metric:theorem}
Let $(X,d)$ be a metric space homeomorphic to $\R^2$ and $\Omega$ be an open subset of $X$. Suppose that $u\colon \Omega \to \R$ is a continuous function with the property that for a.e.\ $t\in \R$ the level set $u^{-1}(t)$ has locally finite Hausdorff $1$-measure (in the metric of $X$). Then for a.e.\ $t\in \R$ each component of the level set $u^{-1}(t)$ is either a point, or a Jordan curve, or it is homeomorphic to an interval. 

If, in addition, $u$ is monotone, then for a.e.\ $t\in \R$ the level set $u^{-1}(t)$ is an embedded $1$-dimensional topological submanifold of $X$.
\end{theorem}

Monotone Sobolev functions enjoy some important further regularity properties. For example, if $f=(u_1,\dots,u_n)\colon \R^n\to \R^n$ is continuous and monotone, in the sense that the coordinate functions $u_i$ are, and $f\in W^{1,n}_{\loc}(\R^n)$, then $f$ is differentiable a.e.\ and it has the \textit{Luzin property}, i.e, it maps null sets to null sets; see \cite[Lemma 1.3]{HajlaszMaly:approximation}. For this reason, the approximation of Sobolev functions $u$ in the Sobolev norm by {locally weakly monotone} Sobolev functions $u_n$, $n\in \N$, has been established in \cite[Theorem 1.3]{HajlaszMaly:approximation}; a property of the approximants $u_n$ in this theorem that is important in applications is that the gradient of $u_n$ vanishes on the set $\{u\neq u_n\}$. As it is pointed out in that paper, in some cases the approximating functions cannot be taken to be smooth, not even continuous. 

As an application of Theorem \ref{Intro:MonotoneSobolev:Theorem} we give the following approximation theorem:

\begin{theorem}\label{Intro:Approximation:Theorem}
Let $\Omega\subset \R^2$ be a bounded open set and $u \colon \Omega\to \R$ be a continuous monotone function  that lies in $W^{1,p}(\Omega)$ for some $1< p <\infty$. Then there exists $\alpha>0$ and a sequence $\{u_n\}_{n\in \N}$ of monotone functions in $\Omega$ such that
\begin{enumerate}[\upshape(A)]
\item $u_n$ is $C^{1,\alpha}$-smooth in $\Omega$, $n\in \N$,
\item $u_n$ converges uniformly to $u$ in $\Omega$ as $n\to\infty$,
\item $u_n$ converges to $ u$ in $W^{1,p}(\Omega)$ as $n\to\infty$,
\item $\|\nabla u_n\|_{L^p(\Omega)}\leq \|\nabla u\|_{L^p(\Omega)}$, $n\in \N$,
\item $u_n-u \in W^{1,p}_0(\Omega)$, $n\in \N$.
\end{enumerate} 
In fact, $u_n$ may be taken to be $p$-harmonic on a subset of $\Omega$ having measure arbitrarily close to full and $C^\infty$-smooth, except at a discrete subset of $\Omega$. If $p=2$, then the functions $u_n$ can be taken to be $C^\infty$-smooth in $\Omega$.
\end{theorem}

Here, a function $f\colon \Omega\to \R$ is $C^{1,\alpha}$-smooth if it has derivatives of first order that are locally $\alpha$-H\"older continuous. 

We remark that standard mollification with a smooth kernel does not produce a monotone function and one has to use the structure of the level sets of the monotone function in order to construct its approximants. Therefore, Theorem \ref{Intro:MonotoneSobolev:Theorem} provides us with a powerful tool in this direction.

For our proof we follow the strategy of \cite{IwaniecKovalevOnninen:W1papproximation} and \cite{IwaniecKovalevOnninen:W12approximation}, where it is proved that Sobolev homeomorphisms of planar domains can be approximated by $C^\infty$-smooth diffeomorphisms uniformly and in the Sobolev norm. In our proof we have to deal with further technicalities, not present in the mentioned results for homeomorphisms, which are related to the fact that our approximants $u_n$ might have critical points. In fact, $u_n$ will be $p$-harmonic in a large subset of $\Omega$ and it is known that at a critical point a $p$-harmonic function, $p\neq 2$, is only expected to be $C^{1,\alpha}$-smooth, rather than $C^{\infty}$-smooth.

The main motivation for Theorem \ref{Intro:Approximation:Theorem} was to study regularity properties of a certain type of infimizers that appear in the setting of Sierpi\'nski carpets and are called \textit{carpet-harmonic functions}; see \cite[Chapter 2]{Ntalampekos:CarpetsThesis}. Namely, these infimizers are restrictions of monotone Sobolev functions (under some geometric assumptions) and the approximation Theorem \ref{Intro:Approximation:Theorem} would imply some absolute continuity properties for these functions. We will not discuss these applications any further in this paper.

We pose some questions for further study. One of the reasons that we were not able to prove approximation by smooth functions for all $p\in (1,\infty)$ in Theorem \ref{Intro:Approximation:Theorem} was the presence of critical points of $p$-harmonic functions. 
\begin{question}
Can a $p$-harmonic function be approximated uniformly and in the $W^{1,p}$ norm by $C^\infty$-smooth monotone functions near a critical point, when $p\neq 2$?
\end{question} 
A positive answer to this question would imply that we may use smooth functions in Theorem \ref{Intro:Approximation:Theorem}. It would be very surprising if this fails.  Moreover, what can one say for higher dimensions?
\begin{question}
Do analogs of Theorems \ref{Intro:MonotoneSobolev:Theorem} and \ref{Intro:Approximation:Theorem} hold in higher dimensions?
\end{question}
The smooth approximation of Sobolev homeomorphisms is still open in higher dimensions (\cite[Question 1.1]{IwaniecKovalevOnninen:W12approximation}, \cite{FoundCompMath:book-Ball}). Since the coordinate functions of a homeomorphism are monotone, a first step in studying this problem would be to study the smooth approximation of monotone Sobolev functions as in the above question.

The paper is structured as follows. In Section \ref{Section:Level} we study the level sets of Sobolev functions. In particular, in Subsection \ref{Section:Sobolev} we prove Theorem \ref{Intro:Sobolev:Theorem} and in Subsection \ref{Section:levelmonotone} we prove Theorem \ref{Intro:MonotoneSobolev:Theorem} and discuss the generalization that gives Theorem \ref{Intro:metric:theorem}. In Subsection \ref{Section:glue} we include some gluing results for monotone functions that are used in the proof of the approximation Theorem \ref{Intro:Approximation:Theorem}, but also are of independent interest. Section \ref{Section:Preliminaries} contains preliminaries on $p$-harmonic functions. The approximation Theorem \ref{Intro:Approximation:Theorem} is proved in Section \ref{Section:approximationTheorem}. Finally, in Section \ref{Section:Smoothing} we include some quite standard smoothing results for monotone functions that are needed for the proof of the approximation theorem.

\subsection*{Acknowledgments}
The author would like to thank Tadeusz Iwaniec for a motivating discussion and Matthew Romney for his comments on the manuscript.

\section{Level sets of Sobolev and monotone functions}\label{Section:Level}

Throughout the entire section we assume that $u\colon \Omega\to \R$ is a non-constant continuous function on an open set $\Omega\subset \R^2$. We define $A_t=u^{-1}(t)$ for $t\in \R$, which can be the empty set. For $ s<t$ we define $A_{s,t}=u^{-1}((s,t))$. 

A \textit{Jordan arc} in a metric space $X$ is the image of the interval $[0,1]$ under a homeomorphic embedding $\phi \colon [0,1] \to X$. In this case, the set $\phi((0,1))$ is called an \textit{open Jordan arc}. Finally, a \textit{Jordan curve} is the image of $S^1$ under a homeomorphic embedding $\phi\colon S^1\to X$. 

\subsection{Level sets of Sobolev functions}\label{Section:Sobolev}

In this subsection we study the level sets of Sobolev functions and prove Theorem \ref{Intro:Sobolev:Theorem}.

\begin{lemma}\label{FiniteLength:Lemma}
Suppose that $u\in W^{1,p}(\Omega)$ for some $1\leq p \leq \infty$ and $\Area(\Omega)<\infty$. Then for a.e.\ $t\in \R$ the level set $A_t$ has finite Hausdorff $1$-measure.
\end{lemma}
\begin{proof}
This follows from the co-area formula \cite[Theorem 1.1]{MalySwansonZiemer:Coarea}, which is attributed to Federer \cite[Section 3.2]{Federer:gmt}, and the $L^p$-integrability of $\nabla u$:
\begin{align*}
\int \mathcal H^1(u^{-1}(t))\, dt= \int_\Omega |\nabla u| \leq  \|\nabla u\|_{L^p(\Omega)} \cdot \Area(\Omega)^{1/p'}<\infty,
\end{align*}
where $\frac{1}{p}+\frac{1}{p'}=1$.
\end{proof}

Now, we restate and prove Theorem \ref{Intro:Sobolev:Theorem}.

\begin{theorem}\label{Jordan:Theorem}
Suppose that $u\in W^{1,p}_{\loc}(\Omega)$ for some $1\leq p\leq \infty$. Then for a.e.\ $t\in \R$ the level set $A_t$ has locally finite Hausdorff $1$-measure and each component $E$ of $A_t$ is either a point, or a Jordan curve, or it is homeomorphic to an interval. 
\end{theorem}

Recall that $u$ is assumed to be continuous. If a level set $A_t=u^{-1}(t)$ is the empty set then it has no components so the conclusions of the theorem hold trivially; this is also true for other statements later in the paper, so we will not mention again the possibility that the level sets can be empty.

\begin{proof}
\textbf{Main Claim:}
Consider an open set $U$ with $U\subset \subset \Omega$.  We restrict $u$ to a neighborhood of $\br U$ that is compactly contained in $\Omega$ and apply Lemma \ref{FiniteLength:Lemma}. It follows that for a.e.\ $t\in \R$ we have $\mathcal H^1(A_t\cap \br U)<\infty$. We shall show as our Main Claim that if we further exclude countably many values $t$, then each component $E_0$ of $A_t\cap \br U$ is either a point, or a Jordan arc, or a Jordan curve. 

\vspace{1em}

\noindent
\textbf{Compact exhaustion argument:} Assuming that the Main Claim holds for each such $U$, we consider an exhaustion of $\Omega$ by a nested sequence of open sets $\{U_n\}_{n\in \N}$, each compactly contained in $\Omega$, such that for a.e.\ $t\in \R$ the following holds for the level set $A_t$: $\mathcal H^1(A_t\cap \br{U_n})<\infty$ and each component $E_n$ of $A_t\cap \br {U_n}$ is either a point, or a Jordan arc, or a Jordan curve, for all $n\in \N$. We let $A_t$ be such a level set and fix $x_0\in A_t$. We will show that the component  $E$ of $A_t$ containing $x_0$ is either a point, or a Jordan curve, or it is homeomorphic to an interval. 

We have $x_0\in U_n$ for all sufficiently large $n$. To simplify our notation, we assume that this holds for all $n\in \N$. Let $E, E_n$ be the component of $A_t, A_t\cap \br{U_n}$, respectively, that contains $x_0$. We have $ E_n \subset  E_{n+1} \subset E$ for each $n\in \N$, which follows from the definition of a connected component (as the largest connected set containing a given point). By the continuity of $u$, $A_t$ is rel.\ closed in $\Omega$, so $E_n$ is compact for all $n\in \N$. If $E\subset U_n$ for some $n\in \N$, then $E_n=E$ and therefore $E$ itself is either a point, or a Jordan curve, or a Jordan arc by the Main Claim. This completes the proof in this case.

Suppose now that $E$ is not contained in $U_n$ for any $n\in \N$. Then $E_n$ necessarily intersects $\partial U_n$ for each $n\in \N$ as we see below using the next lemma, which is a direct consequence of \cite[IV.5, Corollary 1, p.~83]{Newman:Topology}.

\begin{lemma}\label{Separation:Lemma}
Suppose that $F$ is a connected component of a compact set $A$ in the plane. Then for each open set $U\supset F$ there exists an open set $V$ with $F\subset V \subset \subset U$ and $\partial V \cap A=\emptyset$.   
\end{lemma}

In our case, if $E_n \subset U_n$, then by the preceding lemma there exists an open set $V\supset E_n$ with $ V\subset\subset  U_n$ and $\partial V\cap (A_t\cap \br{U_n})=\emptyset$. This implies that $\partial V\cap A_t=\emptyset$. Then $\partial V\cap E=\emptyset$ and it follows that $V\cap E$ is both open and closed in $E$, so $V\cap E=E$ by the connectedness of $E$. Then $E=V\cap E\subset U_n$, which contradicts our assumption that $E\not\subset U_n$ for all $n\in \N$. Therefore, $E_n\cap \partial U_n \neq \emptyset$ for all $n\in \N$.

This implies that $E_n$ cannot be a single point since it also contains $x_0\in U_n$, so $E_n$ is either a Jordan arc, or a Jordan curve for each $n\in \N$, by the Main Claim. Note that $E_{n+1} \supsetneq E_n$ for all $n\in \N$, since $E_{n+1}\cap \partial U_{n+1}\neq \emptyset$ and $\partial U_{n+1}\cap \partial U_n=\emptyset$, as $U_n\subset \subset U_{n+1}$. If $E_n$ is a Jordan curve, then $E_{n+1}$ can be neither a Jordan curve, nor a Jordan arc, as it is strictly larger than $E_n$ (this uses the fundamental fact that $S^1$ is not homeomorphic to $[0,1]$). Therefore, $E_n$ is a Jordan arc for all $n\in \N$, and there exist homeomorphisms $\phi_n\colon [0,1]\to E_n$, $n\in \N$. Since $E_n\subsetneq  E_{n+1}$, these homeomorphisms can be pasted appropriately to obtain a continuous injective map $\phi \colon I\to \bigcup_{n\in \N} E_n \eqqcolon \mathcal E$, where $I\subset \R$ is either $\R$ or $[0,\infty)$ (after a change of variables). Assume that $I=[0,\infty)$; the other case is treated in the same way. The map $\phi$ has the property that $\phi^{-1}(E_n)$ is a compact subinterval of $I$ and $\phi^{-1}(E_n)\subsetneq \phi^{-1}(E_{n+1})$ for $n\in \N$. Moreover, $\phi(I)=\mathcal E$ exits all compact subsets of $\Omega$.

It now remains to show that $\phi$ is a homeomorphism and that $E= \bigcup_{n\in \N} E_n$, concluding therefore that $E$ is homeomorphic to an interval as desired. This is subtle and requires to use the assumption that $\mathcal H^1(A_t\cap \br {U_n}) <\infty$ for all $n\in \N$. 

We first claim that $\phi(s)$ does not accumulate at any point of $\Omega$ as $s\to \infty$. Assume for the sake of contradiction that $\phi(s)$ converges to a point $y_0\in U_{n_0}\subset \Omega$ along a subsequence of $s\to \infty$, for some $n_0\in \N$. Since $\phi(s)$ exits all sets $U_n$, we may find disjoint intervals $[a_n,b_n]$, $n\in \N$, such that $\phi(a_n) \to y_0$ as $n\to\infty$, $\phi([a_n,b_n])\subset \br {U_{n_0}}$, and $\phi(b_n)\in \partial U_{n_0}$ for all $n\in \N$. It follows that $\liminf_{n\to\infty}\diam (\phi([a_n,b_n])) \geq \dist(y_0, \partial U_{n_0})>0$. Since the diameter is a lower bound for the Hausdorff $1$-measure in connected spaces (see \cite[Section 18, p.~18]{Semmes:Hausdorff}), we have
\begin{align*}
\mathcal H^1( A_t\cap \br{U_{n_0}}) \geq \sum_{n\in \N} \mathcal H^1(\phi([a_n,b_n])) \geq \sum_{n\in \N} \diam (\phi([a_n,b_n]))=\infty,
\end{align*}
a contradiction. 

This now implies that $\phi$ is a homeomorphism of $I$ onto $\mathcal E$. Indeed, if $\phi(s_n)=y_n\in \mathcal E$ is a sequence converging to $\phi(s_0)=y_0\in \mathcal E$, then $y_n$ is contained in a compact subset of $\Omega$ for all $n\in \N$ and therefore $s_n,s_0$ lie in a compact subinterval $I_0$ of $I$ for all $n\in \N$. By the injectivity oh $\phi$, $\phi|I_0$ is a homeomorphism, so $s_n\to s_0$, and $\phi^{-1}$ is continuous on $\mathcal E$, as desired. Another implication of the previous paragraph is that $\br{\mathcal E} \setminus \mathcal E$ is contained in $\partial \Omega$ and $\mathcal E=\br{\mathcal E}\cap \Omega=  \br{\mathcal E}\cap E$ is rel.\ closed in $\Omega$ and in $E$.

We will show that $\mathcal E$ is also rel.\ open in $E$ and by the connectedness of $E$ we will have $E=\mathcal E$ as desired. Let $x\in \mathcal E$, so $x\in U_{n_0}$ for some $n_0\in \N$ and consider an open neighborhood $V\subset \subset U_{n_0}$ of $x$. We wish to show that $V\cap E \subset \mathcal E$, after shrinking the neighborhood $V$ if necessary. This will complete the proof that $\mathcal E$ is rel.\ open in $E$.

 If $y_0 \in V\cap E\setminus \mathcal E$, then arguing as we did for the construction of $\mathcal E$ and using the Main Claim, one can construct a set $\mathcal F \subset E$ containing $y_0$ and a homeomorphism $\phi' \colon I' \to \mathcal F$, where $I'=\R$ or $I'=[0,\infty)$. Moreover, the set $\mathcal F$ is, by construction, necessarily disjoint from $\mathcal E$. To see this, assume that they intersect at a point $z_0=\phi(s_0)= \phi'(s_0')$. Then there exist non-trivial closed intervals $I_0\subset I$ and $I_0'\subset I'$ such that $\phi(I_0)$ is the subarc of $\mathcal E$ from $x_0$ to $z_0$ and $\phi'(I_0')$ is the subarc of $\mathcal F$ from $y_0$ to $z_0$. We choose a large $k\in \N$ so that $\phi(I_0)\cup \phi'(I_0')\subset U_k$. Then by the definition of $E_k$ we must have $E_k \supset \phi(I_0)\cup \phi'(I_0')$, since the latter is connected and contains $x_0$. Then we have $y_0\in \phi'(I_0')\subset E_k\subset \mathcal E$, which is a contradiction. Hence, we have indeed $\mathcal F \cap \mathcal E=\emptyset$. 
 
Moreover, as in the construction of $\mathcal E$, since $E$ exits all compact subsets of $\Omega$, the set $\mathcal F$ must also have this property; see the statement right before Lemma \ref{Separation:Lemma}. Therefore,
\begin{align*}
\mathcal H^1(\mathcal F\cap \br {U_{n_0}}) \geq \dist (y_0, \partial U_{n_0}) \geq \dist (\partial V, \partial U_{n_0}) \eqqcolon \delta>0.
\end{align*}
Another property of $\mathcal F$ is that it is rel.\ closed in $\Omega$, for the same reason as $\mathcal E$.

If $V\cap E \setminus (\mathcal E\cup \mathcal F)\neq \emptyset$, by repeating the above procedure we may find sets $\mathcal F_i\subset A_t$, $i=1,2,\dots,N$, with the same properties as $\mathcal F\eqqcolon \mathcal F_1$ so that they are disjoint with each other and with $\mathcal E$, and they intersect $V$. We have
\begin{align*}
\infty > \mathcal H^1(A_t\cap \br{U_{n_0}} ) \geq \sum_{i=1}^N \mathcal H^1(\mathcal F_i\cap \br{U_{n_0}} ) \geq N\delta.
\end{align*}
This implies that we can find only a finite number of such sets $\mathcal F_i$. Since the compact sets $\mathcal F_i\cap \br V$ have positive distance from $\mathcal E\cap \br V$, we may find a smaller neighborhood $W\subset V$ of $x\in \mathcal E$ such that $W\cap E \subset \mathcal E$. This completes the proof that $\mathcal E$ is rel.\ open in $E$, as desired.

\vspace{1em}
\noindent
\textbf{Proof of Main Claim:} 
We will show that for all but countably many $t\in \R$ for which $\mathcal H^1(A_t\cap \br U)<\infty$ we have that each component $E_0$ of $A_t\cap \br U$ is either a point, or a Jordan arc, or a Jordan curve. 

Suppose that $\mathcal H^1(A_t\cap \br U)<\infty$ and let $E_0$ be a component of $A_t\cap \br U$, so $L\coloneqq \mathcal H^1(E_0)<\infty$. Since $E_0$ is a continuum, it follows that there exists a $2L$-Lipschitz continuous parametrization $\gamma\colon [0,1]\to E_0$ of $E_0$; see \cite[Theorem 2a]{EilenbergHarrold:Curves} or \cite[Proposition 5.1]{RajalaRomney:reciprocal}. Hence, $E_0$ is a locally connected, compact set; see \cite[Theorem 9.2, p.~60 and Theorem 27.12, p.~200]{Willard:topology}. 

Now, we need the following topological lemma that we prove later:

\begin{lemma}\label{Junction:Lemma}
Let $X$ be a Peano space, i.e., a connected, locally connected, compact metric space. If $X$ contains more than one points and is not homeomorphic to $[0,1]$ or $S^1$, then it must have a \textit{junction point} $x$, i.e., there exist three Jordan arcs $A_1,A_2,A_3\subset X$ that meet at $x$ but otherwise they are disjoint.
\end{lemma}

By our previous discussion, if $\mathcal H^1(A_t \cap \br U)<\infty$ then each component $E_0$ of $A_t\cap \br U$ is a Peano space. If there is a component $E_0$ of $A_t\cap \br U$ that is not a point or a Jordan arc or a Jordan curve, then by Lemma \ref{Junction:Lemma}, $E_0$ must have a junction point. Hence, $A_t\cap \br U$ has a junction point.

A theorem of Moore \cite[Theorem 1]{Moore:triods} (see also \cite[Proposition 2.18]{Pommerenke:conformal}) states that there can be no uncountable collection of disjoint compact sets in the plane such that each set has a junction point. Note that $A_s\cap A_t=\emptyset$ for $s\neq t$. Hence, for at most countably many $t\in \R$ the set $A_t\cap \br U$ can have a junction point. 

Summarizing, for at most countably many $t\in \R$ for which $\mathcal H^1(A_t\cap \br U)<\infty$ the set $A_t\cap \br U$ has a component $E_0$ that has a junction point and is not a Jordan arc or a Jordan curve.   
\end{proof}

\begin{proof}[Proof of Lemma \ref{Junction:Lemma}]
Suppose that $X$ contains more than one points and is not homeomorphic to $[0,1]$. Then there exist at least three \textit{non-cut points} $x_1,x_2,x_3\in X$, i.e., points whose removal does not disconnect $X$; see \cite[Theorems (6.1) and (6.2), p.~54]{Whyburn:topology}. Since $X$ is a Peano space, it is locally path-connected \cite[Theorem 31.4, p.~221]{Willard:topology}. It follows that each of the spaces $X\setminus \{x_{i}\}$, $i=1,2,3$, is connected and locally path-connected. Moreover, a connected, locally path-connected space is path-connected \cite[Theorem 4.3, p.~162]{Munkres:topology}. Hence, we may find Jordan arcs $J_{ij} \subset X \setminus \{x_k\}$ with endpoints $x_i$ and $x_j$, where $i,j,k\in \{1,2,3\}$, are distinct and $i<j$; see also \cite[Corollary 31.6, p.~222]{Willard:topology}.

If two of the arcs $J_{ij}$ intersect at an interior point, i.e., a point different from $x_1,x_2,x_3$, then it is straightforward to consider three subarcs $A_1,A_2,A_3$ of these arcs that intersect at one point, but otherwise are disjoint, as required in the statement of the lemma.

If the arcs $J_{ij}$ intersect only at the endpoints, then we can concatenate the three arcs to obtain a Jordan curve $J$, i.e., homeomorphic to $S^1$. Suppose now that the space $X$ is, in addition, not homeomorphic to $S^1$. Then, there must exist a point $x\in X\setminus J$. We claim that there exists a Jordan arc $J_x$ that connects $x$ to $J$ and intersects $J$ at one point. Assuming that claim, one can now define $A_1\coloneqq J_x$ and $A_2,A_3$ to be subarcs of $J$ that meet at $x$ but otherwise are disjoint.

To prove the claim note that since $X$ is a {Peano space}, any two points in $X$ can be joined with a Jordan arc \cite[Theorem 31.2, p.~219]{Willard:topology}. We connect $x$ to any point $y\in J$ with a Jordan arc $J_x$, parametrized by $\gamma\colon[0,1]\to J_x$, so that $\gamma(0)=x$ and $\gamma(1)=y$. If there exists $s\in (0,1)$ such that $\gamma(s)\in J$ then we consider the smallest such $s$ and we restrict $\gamma$ to $[0,s]$. This gives the desired Jordan arc. 
\end{proof}

\begin{remark}\label{remark:metric spaces}
In Theorem \ref{Jordan:Theorem} (Theorem \ref{Intro:MonotoneSobolev:Theorem}) the assumption that $u\in W^{1,p}_{\loc}(\Omega)$ is only used to deduce that almost every level set of $u$ has locally finite Hausdorff $1$-measure; the latter is proved  in Lemma \ref{FiniteLength:Lemma} using the co-area formula. All the other arguments rely on planar topology. Thus our proof can be generalized to functions defined on metric spaces homeomorphic to $\R^2$ and the first part of Theorem \ref{Intro:metric:theorem} follows. 
\end{remark}

\subsection{Level sets of monotone functions}\label{Section:levelmonotone}
In this subsection we study the level sets of monotone functions and prove Theorem \ref{Intro:MonotoneSobolev:Theorem}. Recall that $u\colon \Omega\to \R$ is assumed to be continuous.

\begin{lemma}\label{Monotonicity:Corollary}
If $u$ is monotone in $\Omega$ then for each $t\in \R$ each component $E$ of the level set $A_t$ satisfies either of the following: 
\begin{enumerate}[\upshape(i)]
\item $E$ exits all compact subsets of $\Omega$, or
\item all bounded components of $\R^2 \setminus E$ intersect $\partial \Omega$ and there exists at least one such component.
\end{enumerate}
In particular, 
$$\diam (E) \geq \sup_{x\in E} \dist(x,\partial \Omega)>0.$$
\end{lemma}
\begin{proof}

In the proof we will use the following lemma, known as Zoretti's theorem, which is in the same spirit as Lemma \ref{Separation:Lemma}.
\begin{lemma}[{\cite[Corollary 3.11, p.~35]{Whyburn:TopologicalAnalysis}}]\label{Zoretti:Lemma}
Let $F$ be a connected component of a compact set $A$ in the plane. Then for each open set $U \supset F$ there exists a Jordan region $V$ with $F\subset V$, $\partial V\subset U$, and $\partial V\cap A=\emptyset$.  
\end{lemma}

Suppose that a component $E$ of $A_t$ is compactly contained in $\Omega$. First we will show that $\R^2\setminus E$ has a bounded component. Suppose that this is not the case. Then $\widehat\C\setminus E$ is simply connected and contains $\partial \Omega$, so by using the Riemann mapping theorem we may find arbitrarily close to $E$ Jordan curves surrounding $E$ and separating $E$ from $\partial \Omega$. Hence, we may find a Jordan region $U \subset \subset \Omega$ containing $E$. Consider the compact set $A_t\cap \br U$. Then $E$ is also a component of $A_t\cap \br U$. By Lemma \ref{Separation:Lemma}, we can find a Jordan region $V$ such that $E\subset V$, $\partial V\subset U$, and $\partial V\cap A_t=\emptyset$. It follows that $V\subset U\subset \subset \Omega$. On $\partial  V$ we must have $u>t$ or $u<t$. Without loss of generality, suppose that we have $u>t$ on $\partial V$. Then by the monotonicity of $u$ we have $u>t$ on $ V\supset E$, a contradiction. Therefore, $\R^2\setminus E$ has a bounded component.

Next, we will show that all bounded components of $\R^2\setminus E$ must intersect $\partial \Omega$. Suppose that a bounded component $U$ of $\R^2\setminus E$ does not intersect $\partial \Omega$. Then $U\subset \Omega$ and $\partial U \subset E\subset \subset \Omega$ so $U\subset \subset \Omega$. Since $u=t$ on $\partial U$, by the monotonicity of $u$ we have that $u=t$ on $\br U$. Since $E$ is a connected component of $A_t$, we must have $U\subset E$, a contradiction.

Now we prove the claim involving the diameters. If $\Omega=\R^2$, then $E$ necessarily satisfies the first alternative, so it escapes to $\infty $ and $\diam(E)=\infty$. If $\Omega \subsetneq \R^2$, then for each $x\in E$ we may consider the largest ball $B(x,r)$ contained in $\Omega$, where $r=\dist(x,\partial \Omega)$. Then $E$ cannot be contained in a ball $B(x,r-\varepsilon)$ for any $\varepsilon>0$ since this would violate both alternatives. Therefore, $\diam(E) \geq r-\varepsilon$ for all $\varepsilon>0$, which implies that $\diam(E) \geq \dist(x,\partial \Omega)$, as desired.
\end{proof}

We record an immediate corollary:
\begin{corollary}
If $\Omega$ is simply connected, then $u$ is monotone in $\Omega$ if and only if each component of each level set of $u$ exits all compact subsets of $\Omega$.
\end{corollary}
\begin{proof}
Suppose that $u$ is monotone. If $\Omega$ is simply connected, then $\partial \Omega$ is connected, so it cannot be separated by a level set $A_t$ of $u$. Thus, in Lemma \ref{Monotonicity:Corollary} only the first alternative can occur, as desired. 

Conversely, suppose that only the first alternative occurs and let $U\subset\subset \Omega$. Then for any $x_0\in U$ the component $E$ of the level set $A_t$, $t=u(x_0)$, that contains $x_0$ has to intersect $\partial U$. Thus, there exists $y_0\in \partial U$ with $u(x_0)=u(y_0)$. This implies the monotonicity of $u$.
\end{proof}

Next, we add the assumption that $u$ lies in a Sobolev space:

\begin{lemma}\label{MonotoneSobolev:Lemma}
Suppose that $u$ is monotone in $\Omega$ and lies in $W^{1,p}_{\loc}(\Omega)$ for some $1\leq p \leq \infty$. Then for a.e.\ $t\in \R$ the components of the level set $A_t$ are rel.\ open in $A_t$. In other words, if $E$ is a component of $A_t$ and $x\in E$ then there exists an open neighborhood $U$ of $x$ such that $E\cap U= A_t\cap U$. 
\end{lemma}
\begin{proof}
We consider an exhaustion of $\Omega$ by a nested sequence of open sets $\{U_n\}_{n\in \N}$, each compactly contained in $\Omega$, such that for a.e.\ $t\in \R$ we have $\mathcal H^1(A_t\cap \br{U_n})<\infty$ for all $n\in \N$; the existence of such an exhaustion can be justified using Lemma \ref{FiniteLength:Lemma}. 

We fix $t\in \R$ such that $A_t\neq \emptyset$, and consider a component $E\subset A_t$ and $x\in E$. We claim that for each neighborhood $V\subset \subset \Omega$ of $x$ there are at most finitely many components of $A_t$ intersecting $V$. 

There exists $n_0\in \N$ such that $V\subset \br V\subset   U_{n_0}$. Suppose that $F$ is a component of $A_t$ intersecting $V$. We consider the restriction of $u$ to $U_{n_0}$, which is still a monotone function and we let $G\subset F$ be a component of the level set $A_t\cap U_{n_0}$ of $u\big|{U_{n_0}}$ such that $G\cap V\neq \emptyset$. For a point $y\in G\cap V$ we have 
$$\diam(G) \geq \dist(y,\partial U_{n_0})$$
by Lemma \ref{Monotonicity:Corollary}. Since $V\subset\subset U_{n_0}$, we have $\dist(y,\partial U_{n_0}) \geq \dist(\br V, \partial U_{n_0}) >0$. Moreover, by the connectedness of $G$ we have $\mathcal H^1(G) \geq \diam (G)$; see \cite[Section 18, p.~18]{Semmes:Hausdorff}. Combining these inequalities, we have
\begin{align*}
\mathcal H^1(F\cap \br{U_{n_0}} )\geq \mathcal H^1(G)\geq \diam(G)\geq \dist(y,\partial U_{n_0})\geq \dist(\br V, \partial U_{n_0}) >0.
\end{align*}
Since $\mathcal H^1(A_t\cap\br{U_{n_0}})<\infty$, there can be at most finitely many components $F$ of $A_t$ intersecting $V$.

Since the compact sets $\br V\cap F$ and $\br V\cap E$ have positive distance, it follows that if we choose a smaller neighborhood $U\subset V$ of $x$, then we have $E\cap U=A_t\cap U$ as desired. 
\end{proof}

We will also need the following general lemma:

\begin{lemma}\label{Extremal:Lemma}
Let $(X,d)$ be a separable metric space and $f\colon X \to \R$ be any function. Then the set of local extremal values of $f$ is at most countable. 
\end{lemma}
\begin{proof}
Let $E$ be the set of local maximum values of $f$. Then, by definition, for each $y\in E$ there exists $x\in X$ with $f(x)=y$ and a ball $B(x,r)$ such that for all $z\in B(x,r)$ we have $f(z)\leq y$. We can write 
$E=\bigcup_{n\in \N} E_n$, where 
$$E_n=\{y\in \R: y=f(x) \,\,\,\textrm{for some}\,\,\, x\in X \,\,\, \textrm{and}\,\,\, f(z)\leq y \,\,\, \textrm{for all}\,\,\, z\in B(x,2/n)\}.$$
We will show that $E_n$ is at most countable for each $n\in \N$. Let $y_1,y_2\in E_n$ be distinct, so there exist points $x_1,x_2\in X$ with $f(x_i)=y_i$, $i=1,2$, as in the definition of $E_n$. Then the balls $B(x_1, 1/n)$, $B(x_2,1/n)$ are necessarily disjoint. Indeed, otherwise, we have $f(x_2)\leq y_1$ and $f(x_1)\leq y_2$, so $y_1=y_2$, a contradiction. Therefore, the set $E_n$ is in one-to-one correspondence with a collection of disjoint balls in $X$. The separability of $X$ implies that there can be at most countably many such balls. The same proof shows that the set of local minimum values of $f$ is at most countable.
\end{proof}

Now we have completed the preparation for the proof of Theorem \ref{Intro:MonotoneSobolev:Theorem}, restated below:

\begin{theorem}\label{Sard:Theorem}
Suppose that $u$ is monotone in $\Omega$ and lies in $W^{1,p}_{\loc}(\Omega)$ for some $1\leq p \leq \infty$. Then for a.e.\ $t\in \R$ the level set $A_t$ is an embedded $1$-dimensional topological submanifold of $\R^2$ that has locally finite Hausdorff $1$-measure.
\end{theorem}

\begin{proof}
By Theorem \ref{Jordan:Theorem} and Lemma \ref{MonotoneSobolev:Lemma} we conclude that for a.e.\ $t\in \R$ the level set $A_t$ has locally finite Hausdorff $1$-measure, each component $E$ of the level set $A_t$ is rel.\ open in $A_t$, and it is either a point, or a Jordan curve, or it is homeomorphic to an interval. Using Lemma \ref{Extremal:Lemma}, we further exclude the countably many local extremal values $t\in \R$ of $u$ in $\Omega$. We fix a level set $A_t$ satisfying all the above. In particular, $A_t$ has the property that if $x\in A_t$ then each neighborhood $U$ of $x$ contains points with $y_1,y_2$ with $u(y_1)>t$ and $u(y_2)<t$. 

Our goal is to prove that every component $E$ of $A_t$ is either a Jordan curve, or it is homeomorphic to $\R$. Since each component of $A_t$ is rel.\ open in $A_t$, it will then follow that each $x\in A_t$ has a neighborhood $U$ such that $U\cap A_t$ is homeomorphic to an open interval. This shows that $A_t$ is a $1$-manifold. Since there are no wild arcs in the plane (see Remark \ref{Straighten:Remark} below), each Jordan arc of $A_t$ can be mapped to $[0,1]\times \{0\}$ with a global homeomorphism of $\R^2$. This shows that $A_t$ is an embedded submanifold of $\R^2$, as desired. 

We now focus on proving that every component $E$ of $A_t$ is either a Jordan curve, or it is homeomorphic to $\R$

We already know that $E$ is either a point, or a Jordan curve, or it is homeomorphic to an interval. If $E$ is a point or it is homeomorphic to the closed interval $[0,1]$, then $E$ is compactly contained in $\Omega$, so by Lemma \ref{Monotonicity:Corollary} the set $\R^2\setminus E$ must have at least one bounded component. This is a contradiction.

Finally, assume that $E$ is homeomorphic to $I=[0,\infty)$ under a map $\phi\colon I \to E$. Then $\phi(s)$ cannot accumulate to any point of $\Omega$ as $s\to\infty$. This is because $\phi$ is a homeomorphism and $E$ is rel.\ closed in $A_t$, and thus in $\Omega$. Let $x_0=\phi(0)$ and consider a ball $B(x_0,r)\subset \subset \Omega$. Then there exists $s_0$ such that $\phi(s_0)\in \partial B(x_0,r)$ and $\phi(s)\notin B(x_0,r)$ for all $s\geq s_0$. We may straighten $\phi([0,s_0])$ to the segment $[0,1]\times \{0\}$ with a homeomorphism of $\R^2$ (see Remark \ref{Straighten:Remark} below) and, using that, we can find a topological ball $U\subset B(x_0,r)$ containing $x_0$ such that $U\setminus E$ is connected. If $U$ is sufficiently small, then by Lemma \ref{MonotoneSobolev:Lemma} we have that $U\setminus A_t=U\setminus E$, so $U\setminus A_t$ is connected. This implies that $u>t$ or $u<t$ in $U\setminus E$. This is a contradiction, since $u(x_0)=t$ would then be a local extremal value of $u$ in this case.
\end{proof}

\begin{remark}\label{Straighten:Remark}
In the proof we used the fact that a Jordan arc in $\R^2$ can be mapped to $[0,1]\times \{0\}$ with a homeomorphism of $\R^2$. That is, there are no \textit{wild} arcs in the plane. To see this, one can first embed the Jordan arc in a Jordan curve, and then apply the Schoenflies theorem to straighten the Jordan curve. See also the discussion in \cite{NoWildArcs}. 
\end{remark}

\begin{remark}\label{remark:metric spaces monotone}
As in Remark \ref{remark:metric spaces} the proof of Theorem \ref{Sard:Theorem} (Theorem \ref{Intro:MonotoneSobolev:Theorem}) can be generalized to monotone functions defined on metric spaces homeomorphic to $\R^2$. Indeed, monotonicity is a topological property. This observation yields the second part of Theorem \ref{Intro:metric:theorem}.
\end{remark}

\subsection{Gluing monotone functions}\label{Section:glue}
In this section we include some results that allow us to paste monotone functions in order to obtain a new monotone function. These results will be useful towards the proof of the approximation Theorem \ref{Intro:Approximation:Theorem}. The proofs are elementary but the assumptions of the statements are finely chosen and cannot be relaxed. Recall that $u$ is continuous in $\Omega$.

\begin{lemma}[Gluing Lemma]\label{Gluing:Lemma}
Suppose that $u$ is monotone in $\Omega$ and consider $t_1,t_2\in \R$ with $t_1<t_2$. Let $\Upsilon=u^{-1}((t_1,t_2))$ and consider a continuous function $v$ on $\Upsilon$ such that 
\begin{enumerate}[\upshape(i)]
\item $v$ is monotone in $\Upsilon$,
\item  $v$ extends continuously to $\partial \Upsilon\cap \Omega$ and agrees there with $u$, and
\item  $t_1\leq v \leq t_2$ on $\Upsilon$.
\end{enumerate}
Then the function $\widetilde u$ that is equal to $u$ in $\Omega\setminus \Upsilon$ and to $v$ in $\Upsilon$ is monotone in $\Omega$. 
\end{lemma}
The proof we give below is elementary but delicate. Note that it is important to assume that $u$ is monotone in all of $\Omega$, since the function $u(x)=|x|$ on $\{x\in \R^2: 1/2<|x|<1\}$ does not have a monotone extension in the unit disk. Moreover, (iii) cannot be relaxed. Indeed, the function $u(x)=|x|$ is monotone in the punctured unit disk, considered to be the set $\Omega$; however if we set $\widetilde u=u$ in $\{x\in \R^2: 1/2\leq |x|<1\}$ and $\widetilde u = 1-|x|$ in $\{x\in \R^2 : 0<|x|<1/2\}=u^{-1}((0,1/2))$, then $\widetilde u$ is not monotone in $\Omega$. 

\begin{proof}
Aiming for a contradiction, suppose that $\widetilde u$ is not monotone in $\Omega$, so, without loss of generality, there exists an open set $U\subset \subset \Omega$ and $x_0\in U$ such that $\widetilde u(x_0)=\max_{\br U} \widetilde u >\max_{\partial U} \widetilde u$. 

Suppose first that $x_0\in U\cap \Upsilon$, so $\widetilde u(x_0)=v(x_0)$. Note that $\br{U\cap \Upsilon} \subset \br \Upsilon \cap \Omega$, so $v$ is continuous in $\br{U\cap \Upsilon} \subset \br \Upsilon \cap \Omega$ by assumption (ii) and monotone in $U\cap \Upsilon$ by (i). By Remark \ref{Monotone:Remark} we conclude that there exists $y_0\in \partial (U\cap \Upsilon) \subset (U\cap \partial \Upsilon)\cup \partial U$ such that 
$$\widetilde u(y_0)=v(y_0)=\max_{\br{U\cap \Upsilon}}v \geq v(x_0)= \widetilde u(x_0)= \max_{\br U}\widetilde u.$$
It follows that $\widetilde u(y_0)= \widetilde u(x_0)=\max_{\br U}\widetilde u> \max_{\partial U}\widetilde u$. Hence, $y_0\notin \partial U$, and we have $y_0\in U\cap \partial \Upsilon$. Since $\partial \Upsilon \cap \Omega \subset A_{t_1}\cup A_{t_2}$, we have $\widetilde u(y_0)=u(y_0)=t_1$ or $\widetilde u(y_0)=u(y_0)=t_2$. 

By the monotonicity of $u$ in $\Omega$, it follows that there exists $z_0\in \partial U$ such that 
$$u(z_0)=\max_{\br U}u\geq u(y_0).$$ 
If  $z_0\notin \Upsilon$, then $\widetilde u(z_0)=u(z_0)$, so $\widetilde u(z_0)\geq u(y_0)=\widetilde u(y_0)$. If $z_0\in \Upsilon$, then $\max_{\br U}u=u(z_0)<t_2$ so on $U\setminus \Upsilon \supset U\cap \partial \Upsilon$ we necessarily have $u\leq t_1$. Since $y_0\in U\cap \partial \Upsilon$, we have $\widetilde u(y_0)=u(y_0)=t_1$. Moreover, $\widetilde u(z_0)=v(z_0)$ and by assumption (iii) we have $\widetilde u(z_0)\geq t_1$. It follows that $\widetilde u(z_0)\geq \widetilde u(y_0)$ also in this case. Therefore, 
$$\max_{\br U} \widetilde u= \widetilde u(y_0) \leq \widetilde u(z_0) \leq \max_{\partial U} (\widetilde u),$$
which is a contradiction. 

It remains to treat the case that $x_0\in U\setminus \Upsilon$, so $\widetilde u(x_0)=u(x_0) \notin (t_1,t_2)$.  If $u(x_0)\leq t_1$, then $\max_{\partial U}\widetilde u < u(x_0)\leq t_1$, so $\widetilde u<t_1$ on $\partial U$. By assumption (iii) we have $\partial U\cap \Upsilon =\emptyset$, so $u=\widetilde u$ on $\partial U$, and $\max_{\partial U} u =\max_{\partial U} \widetilde u <u(x_0)$, where $x_0\in U$. This contradicts the monotonicity of $u$ in $\Omega$. If $u(x_0)\geq t_2$, by the monotonicity of $u$ in $\Omega$, there exists $z_0\in \partial U$ such that $u(z_0)=\max_{\br U}u \geq u(x_0) \geq t_2$. This implies that $z_0\notin \Upsilon$, so $\widetilde u(z_0)=u(z_0)$. Therefore,
$$\max_{\partial U} \widetilde u \geq \widetilde u(z_0) =u(z_0)\geq u(x_0)= \widetilde u(x_0) =\max_{\br U} \widetilde u,$$
which is a contradiction.
\end{proof}

\begin{lemma}\label{ConvergenceMonotonicity:Lemma}
Let $\{u_n\}_{n\in \N}$ be a sequence of monotone functions in $\Omega$ converging locally uniformly to a function $u$. Then $u$ is monotone in $\Omega$. 
\end{lemma}
\begin{proof}
Let $U\subset \subset \Omega$.  Then $\max_{\br U}u_n =\max_{\partial U}u_n$. Since $u_n\to u$ uniformly in $\br U$, we have $\max_{\br U}u_n \to \max_{\br U} u$ and $\max_{\partial U}u_n \to \max_{\partial U}u$. The claim for the minima is proved in the same way.
\end{proof}

\begin{corollary}\label{Gluing:Corollary}
Suppose that $u$ is monotone in $\Omega$ and consider a bi-infinite sequence of real numbers $\{t_i\}_{i\in \Z}$, such that $t_i<t_{i+1}$, $i\in \Z$, and $\lim_{i\to \pm \infty} t_i= \pm \infty$. In each region $\Upsilon_i\coloneqq u^{-1}((t_i,t_{i+1}))$ consider a function $v_i$ such that
\begin{enumerate}[\upshape(i)]
\item $v_i$ is monotone in $\Upsilon_i$,
\item $v_i$ extends continuously to $\partial \Upsilon_i\cap \Omega$ and agrees there with $u$, and
\item $t_i\leq v_i\leq t_{i+1}$ on $\Upsilon_i$.
\end{enumerate}
Then the function $\widetilde u$ that is equal to $u$ on $\Omega\setminus \bigcup_{i\in \Z} \Upsilon_i$ and to $v_i$ in $\Upsilon_i$, $i\in \Z$, is continuous and monotone in $\Omega$. 
\end{corollary}

\begin{proof}
By Lemma \ref{Gluing:Lemma} and induction, one can show that for each $n\in \N$ the function 
$$\widetilde u_n \coloneqq u \cdot\x_{\Omega\setminus \bigcup_{|i|\leq n} \Upsilon_i}+ \sum_{|i|\leq n} v_i \cdot \x_{\Upsilon_i}$$
is continuous and monotone in $\Omega$. The important observation here is that $$u^{-1}((t_j,t_{j+1}))=\Upsilon_j= \widetilde u_n^{-1}((t_j,t_{j+1}))$$ for all $n\in \N$ and $|j|>n$, by assumption (iii). 

If $U\subset \subset \Omega$ is an open set, then there exists $n\in \N$ such that
$$t_{-n} < \min_{\br U} u \leq \max_{\br U} u <t_n.$$
Hence, $\Upsilon_j\cap \br U=\emptyset$ for $|j|> n$ and $\widetilde u_m=\widetilde u_n=\widetilde u$ on $\br U$ for all $m\geq n$. It follows that $\{\widetilde u_n\}_{n\in \N}$ converges locally uniformly in $\Omega$ to $\widetilde u$, and therefore $\widetilde u$ is continuous and monotone by Lemma \ref{ConvergenceMonotonicity:Lemma}.
\end{proof}

In order to establish the approximation by $C^{\infty}$-smooth functions in Theorem \ref{Intro:Approximation:Theorem}, we need to introduce the notion of strict monotonicity and prove some further, more specialized, gluing lemmas. 

\begin{definition}
A continuous function $f\colon \Omega\to \R$ is called strictly monotone if it is monotone and for each open set $U\subset\subset \Omega$ the maximum and minimum of $f$ on $\br U$ are not attained at any point of $U$.
\end{definition}

\begin{example}
If a function $f$ is of class $C^1$ and has no critical points in $\Omega$, then it has no local maxima and minima in $\Omega$ so it is strictly monotone.
\end{example}

\begin{lemma}[Gluing Lemma]\label{GluingMonotonicity:Lemma}
Let $A,V$ be open subsets of $\Omega$ with $\br V\cap \Omega\subset   A$. Suppose that the function $u$ is monotone when restricted to $\Omega \setminus \br V$ and strictly monotone when restricted to $A$. Then $u$ is monotone in $\Omega$.  
\end{lemma}
\begin{proof}
If the statement fails, there exists an open set $U\subset \subset \Omega$ such that the maximum or minimum of $u$ in $U$ is not attained at $\partial U$, but at an interior point $x_0\in U$. Without loss of generality, assume that $\max_{\partial U} u < \max_{\br U}u =u(x_0)$. Note that $U$ cannot be contained in $\Omega\setminus \br V$ or in $A$, by the monotonicity of $u$ there. Hence, $U$ intersects both $\br V$ and $\Omega\setminus \br A$.

If $x_0\in U\cap \br V \subset U\cap A$, then $u(x_0)= \max_{\br U}u\geq \max_{\br{U\cap A}}u$ and this contradicts the strict monotonicity of $u$ in $A$.

If $x_0\in U\setminus \br V\subset \Omega\setminus \br V$, then by the monotonicity of $u$ there, there exists $x_1\in \partial (U\setminus \br V)$ such that $u(x_1)\geq u(x_0)> \max_{\partial U}u$. We necessarily have that $x_1\notin \partial U$. Since $\partial (U\setminus \br V)\subset \partial U \cup (U\cap \partial V)$, it follows that $x_1\in U\cap \partial V \subset U\cap A$. Then we have $u(x_1)\geq  u(x_0) = \max_{\br U}u \geq \max_{\br{U\cap A}}u$. Again, this contradicts the strict monotonicity of $u$ in $A$.
\end{proof}

\begin{lemma}\label{GluingConstant:Lemma}
Suppose that $J$ is a connected closed subset of $\Omega$ that exits all compact subsets of $\Omega$. If $u$ is monotone in $\Omega\setminus  J$ and $u$ is constant in $J$, then $u$ is monotone in $\Omega$.
\end{lemma}

\begin{proof}
Assume that $u=t$ in $J$. Suppose that $u$ is not monotone in $\Omega$, so, without loss of generality, we can find an open set $U\subset\subset \Omega$ and a point $x_0\in U$ with $u(x_0)=\max_{\br U}u >\max_{\partial U}u$. By the monotonicity of $u$ in $\Omega\setminus J$, we necessarily have that $U$ intersects both $\Omega\setminus J$ and $J$, and $x_0\in J \cap U$. Thus, $u(x_0)=t=\max_{\br U}u$. Since $J$ is connected and is not contained in $\br U$, we have $J\cap \partial U\neq \emptyset$. Hence, there exists $y_0\in \partial U$ such that $u(y_0)=t\leq \max_{\partial U}u$, a contradiction.
\end{proof}

\section{Preliminaries on $p$-harmonic functions}\label{Section:Preliminaries}

Let $\Omega\subset \R^2$ be an open set. A function $u\colon \Omega\to \R$ is called $p$-harmonic, $1<p<\infty$, if $u\in W^{1,p}_{\loc}(\Omega)$ and 
\begin{align*}
\Delta_pu \coloneqq \textrm{div}(|\nabla u|^{p-2} \nabla u) =0
\end{align*}
in the sense of distributions. That is, 
\begin{align*}
\int_\Omega  \langle |\nabla u|^{p-2} \nabla u , \nabla \phi \rangle =0
\end{align*}
for all $\phi \in C^\infty_c(\Omega)$, i.e., compactly supported smooth functions in $\Omega$. 

We mention some standard facts about $p$-harmonic functions. There exists an exponent $\alpha \in (0,1]$ such that every $p$-harmonic function $u$ on $\Omega$ lies in $C^{1,\alpha}_{\loc}(\Omega)$ \cite{Uralceva:C1alpha}, \cite{Evans:C1alpha}, \cite{Lewis:C1alpha}. In fact, outside the \textit{singular set}, i.e., the set where $\nabla u=0$, the function $u$ is $C^\infty$-smooth by elliptic regularity theory; see \cite[Corollary 8.11, p.~186]{GilbargTrudinger:pde}. The singular set consists of isolated points, unless $u$ is constant \cite[Corollary 1]{Manfredi:pharmonicplane}. Finally, the maximum principle \cite[p.~111]{Heinonenetal:DegenerateElliptic} implies that $p$-harmonic functions are monotone.

\begin{prop}[Solution to the Dirichlet Problem]\label{DirichletProblem:Prop}
Let $\Omega\subset \R^2$ be a bounded open set and let $u_0\in W^{1,p}(\Omega)$ be given Dirichlet data. There exists a unique $p$-harmonic function $u\in W^{1,p}(\Omega)$ that minimizes the $p$-harmonic energy
\begin{align*}
E_p(v)\coloneqq \int_\Omega |\nabla  v|^p
\end{align*}
among all functions $v\in W^{1,p}(\Omega)$ with $v-u_0 \in W^{1,p}_0(\Omega)$.  
\end{prop}

See e.g.\ \cite[Chapter 5]{Heinonenetal:DegenerateElliptic} for a general approach.

An open Jordan arc $J$ in $\R^2$ is $C^\infty$-\textit{smooth} if there exists an open set $U\supset J$ and a $C^\infty$-smooth diffeomorphism $\phi \colon U\to \R^2$ such that $\phi(J)=\R$. In this case, the open set $U\supset J$ may be taken to be arbitrarily close to $J$; see also the classification of $1$-manifolds in Theorem \ref{Classification:theorem}.

Let $\Omega\subset \R^2$ be an open set and $J\subset \partial \Omega$ be a Jordan arc. We say that $J$ is an \textit{essential} boundary arc if for each $x_0\in J$ there exists a neighborhood $U$ of $x_0$ and a homeomorphism $\phi\colon U \to \R^2$ such that $\phi(J\cap U)=\R=\{(s,0)\in \R^2:s\in \R\}$ and $\phi(\Omega\cap U)= \R^2_+= \{(s,t)\in \R^2: t>0\}$.  

\begin{lemma}[Continuous extension]\label{ContinuousExtension:Lemma}
Suppose that $u$ and $u_0$ are as in Proposition \ref{DirichletProblem:Prop}. Assume further that $J\subset \partial \Omega$ is an essential open Jordan arc. If $u_0$ extends continuously to $\Omega\cup J$, then $u$ also extends continuously to $\Omega\cup J$ and $u|J=u_0|J$.   
\end{lemma}

This follows from \cite[Theorem 6.27]{Heinonenetal:DegenerateElliptic}, which implies that each point $x_0\in J$ is a \textit{regular point} for the $p$-Laplace operator, since $\partial \Omega$ is \textit{$p$-thick} at $x_0$; see \cite[Lemma 2]{Lehrback:CapacityEstimate} for a relevant capacity estimate.

\begin{lemma}[Comparison]\label{Comparison:Lemma}
Suppose that $u$ and $u_0$ are as in Proposition \ref{DirichletProblem:Prop}. If there exists $M\in \R$ such that $u_0\leq M$ in $\Omega$, then $u\leq M$ in $\Omega$.
\end{lemma}

To see this, one can take an exhaustion of $\Omega$ by regular open sets $\Omega_n$ (see \cite[Corollary 6.32]{Heinonenetal:DegenerateElliptic}) and solve the $p$-Laplace equation in $\Omega_n$ with boundary data $u_{0,n}$ that smoothly approximates in $W^{1,p}(\Omega_n)$ the function $u_0$ and satisfies $u_{0,n}\leq M+1/n$. The solution $u_n$ of the $p$-Laplace equation satisfies $u_n=u_{0,n}$ on $\partial \Omega_n$, and by the maximum principle \cite[p.~111]{Heinonenetal:DegenerateElliptic} we have $u_n\leq M+1/n$ in $\Omega_n$. Now, by passing to a limit \cite[Theorem 6.12]{Heinonenetal:DegenerateElliptic}, we obtain a $p$-harmonic function $\widetilde u\leq M$ that necessarily solves the Dirichlet Problem in $\Omega$. Since the solution is unique, we have $u=\widetilde u\leq M$.

\begin{prop}[Non-degeneracy up to boundary]\label{NonDegeneracy:Prop}
Suppose that $u$ is a $p$-harmonic function in $\Omega$. Moreover, suppose that $J\subset \partial \Omega$ is a $C^\infty$-smooth essential open Jordan arc on which $u$ extends continuously and is equal to a constant $t$, and suppose that $u\neq t$ in a neighborhood of $J$ in $\Omega$. Then $|\nabla u| \neq 0$ in a neighborhood of $J$ in $\Omega$ and in fact each point $x_0\in J$ has a neighborhood in $\Omega$ in which $|\nabla u|$ is bounded away from $0$.   
\end{prop}

This result follows from \cite[Theorem 2.8]{LewisNystrom:BoundaryRegularityPHarmonic} and \cite[Corollary 6.2]{AikawaKilpelainenShanmugalingamZhong:BoundaryHarnack}. The first result implies that each $x_0\in J$ has a neighborhood $B(x_0,r)$ such that $|\nabla u(x)| \geq c |u(x)-t|/\dist(x,\partial \Omega)$ for all $x\in B(x_0,r)\cap \Omega$, where $c>0$ is a constant depending on $p$ and on the geometry of $J$. The second cited result implies that $|u(x)-t|/ \dist(x,\partial \Omega) \geq c'$ for all $x$ in a possibly smaller ball $B(x_0,r')\cap \Omega$, where $c'>0$ is some constant depending on $u$. 
 
If $J\subset \partial \Omega$ is a $C^\infty$-smooth essential open Jordan arc, then a function $u\colon \Omega \to \R$ is said to be ($C^\infty$-)\textit{smooth up to $J$} if there exists an open set $U\supset J$ and $u$ extends to a $C^\infty$-smooth function on $U$. Note that this does \textit{not} require that $u$ is smooth in all of $\Omega$. If $u$ is also smooth in $\Omega$, then we write that $u$ is smooth in $\Omega\cup J$.

\begin{prop}[Boundary regularity]\label{BoundaryRegularity:Prop}
Suppose that $u$ is a $p$-harmonic function in $\Omega$. Moreover, suppose that $J\subset \partial \Omega$ is a $C^\infty$-smooth essential open Jordan arc and that $u$ is $C^\infty$-smooth when restricted to $J$. If each $x_0\in J$ has a neighborhood in $\Omega$ in which $|\nabla u|$ is bounded away from $0$, then $u$ is $C^{\infty}$-smooth up to the arc $J$.
\end{prop}

This follows from \cite[Theorem 6.19, p.~111]{GilbargTrudinger:pde}, since under the assumptions the $p$-harmonic equation is uniformly elliptic up to each compact subarc of the Jordan arc $J$.

Combining Propositions \ref{NonDegeneracy:Prop} and \ref{BoundaryRegularity:Prop} we have:

\begin{corollary}\label{NonDegeneracy:Corollary}
Suppose that $u$ is a $p$-harmonic function in $\Omega$. Moreover, suppose that $J\subset \partial \Omega$ is a $C^\infty$-smooth essential open Jordan arc on which $u$ extends continuously and is equal to a constant $t$, and suppose that $u\neq t$ in a neighborhood of $J$ in $\Omega$. Then $u$ is $C^\infty$-smooth up to the arc $J$ and each point of $J$ has a neighborhood in $\Omega\cup J$ in which $|\nabla u|$ is bounded away from $0$. 
\end{corollary}

\section{Proof of the approximation Theorem}\label{Section:approximationTheorem}

In this section we prove Theorem \ref{Intro:Approximation:Theorem}. Let us start with an elementary lemma:

\begin{lemma}\label{EssentialArcs:Lemma}
Let $f\colon \Omega\to \R$ be a continuous function and suppose that $\{s_i\}_{i\in \Z}$ is a bi-infinite sequence of real numbers with $s_i<s_{i+1}$, $i\in \Z$, and $\lim_{i\to\pm\infty} s_i= \pm\infty$. If each of the level sets $f^{-1}(s_i)$, $i\in \Z$, is an embedded $1$-dimensional topological submanifold of $\R^2$, then
\begin{enumerate}[\upshape(i)]
\item $\bigcup_{i\in \Z} f^{-1}(s_i)$ is also an embedded $1$-dimensional topological submanifold, \vspace{.3em}
\item $\partial f^{-1}((s_j,s_{j+1}))\cap \Omega \subset f^{-1}(s_j) \cup f^{-1}(s_{j+1})$ and $f^{-1}(s_j)\subset\partial f^{-1}((s_{j-1},s_{j})) \cup \partial f^{-1}((s_j,s_{j+1})) $ for all $j\in \Z$, and \vspace{.3em}
\item if $x\in  f^{-1}(s_j)$ for some $j\in \Z$, then there exist disjoint Jordan regions $V_+, V_-$ such that the closure of $V_+\cup V_-$ contains a neighborhood of $x$ and there exists an open Jordan arc $J\subset \partial V_+\cap \partial V_- \cap  f^{-1}(s_j)$ containing $x$ such that $J$ is an essential boundary arc of $ V_+$ and $ V_-$. Moreover, each of $V_+,V_-$ is contained in a component of $  f^{-1}((s_{j-1},s_j))\cup  f^{-1}((s_j,s_{j+1}))$. 
\end{enumerate}
\end{lemma}

\begin{proof}
Let $x\in \bigcup_{i\in \Z} f^{-1}(s_i)$, so $x\in f^{-1}(s_j)$ for some $j\in \Z$. Then we claim that there exists a neighborhood $V$ of $x$ such that $V\cap \bigcup_{i\in \Z} f^{-1}(s_i)= V\cap f^{-1}(s_j)$. This claim trivially implies (i). Note that $x$ has positive distance from $f^{-1}(s_i)$, for all $i\neq j$ by the continuity of $f$. If the claim were not true, then one would be able to find a sequence $x_n\in f^{-1}(s_{i_n})$, where $i_n$ are distinct integers with $i_n\neq j$, $n\in \N$, such that $x_n\to x$ as $n\to \infty$. By continuity $f(x_n)\to f(x)$ so $s_{i_n}\to s_j$ as $n\to \infty$. This contradicts the assumptions on the sequence $\{s_i\}_{i\in \Z}$. 

Part (ii) follows by the continuity of $f$. Indeed, if $x_n \in f^{-1}((s_j,s_{j+1}))$ converges to $x \in \partial f^{-1}((s_j,s_{j+1}))\cap \Omega$, then $f(x_n)$ converges to $f(x)$. Note that $f(x)$ cannot lie in $(s_j,s_{j+1})$, otherwise $f(y)\in (s_j,s_{j+1})$ for all $y$ in a neighborhood of $x$ by continuity, so $x$ would not be a boundary point of $f^{-1}((s_j,s_{j+1}))$.  Therefore $f(x)=s_j$ or $f(x)=s_{j+1}$. For the second claim in (ii) note that $f^{-1}(s_j)$ has empty interior (as a subset of the plane), since it is a $1$-manifold. Therefore for each point $x\in f^{-1}(s_j)$ we can find a sequence of points $x_n$ converging to $x$ with $f(x_n)\neq s_j$, so $f(x_n)\in (s_{j-1},s_j)\cup (s_j,s_{j+1})$ for all sufficiently large $n$. The claim follows.

For (iii), let $x\in f^{-1}(s_j)$. By part (i), we conclude that there exists a neighborhood $V$ of $x$ that does not intersect $f^{-1}(s_i)$ for any $i\neq j$. Since $f^{-1}(s_j)$ is an embedded $1$-dimensional manifold, after shrinking $V$ if necessary, there exists a homeomorphism $\phi\colon V \to \R^2$ such that $\phi(f^{-1}(s_j)\cap V)=\R$; see Remark \ref{Straighten:Remark}. Consider small Jordan regions $ V_+ \subset \phi^{-1}(\R^2_+)$, $ V_- \subset \phi^{-1}(\R^2_-)$ such that the closure of $V_+\cup V_-$ contains a neighborhood of $x$. It follows that there exists an open Jordan arc $J\subset \partial V_+\cap \partial V_- \cap f^{-1}(s_j)$ containing $x$ such that $J$ is an essential boundary arc of $ V_+$ and $ V_-$. By (ii) we have that $x\in \partial f^{-1}((s_{j-1},s_{j}))\cup \partial f^{-1}((s_j,s_{j+1}))$. Hence, $V_+,V_-$ must intersect the open set $f^{-1}((s_{j-1},s_{j}))\cup f^{-1}((s_j,s_{j+1}))$. Since $V_+$ and $V_-$ are connected and they do not intersect the boundary of that set (by the choice of $V$), it follows that they are contained in connected components of $f^{-1}((s_{j-1},s_{j}))\cup f^{-1}((s_j,s_{j+1}))$.
\end{proof}

We let $u\colon \Omega\to \R $ be a continuous monotone function with $u\in W^{1,p}(\Omega)$, as in the statement of Theorem \ref{Intro:Approximation:Theorem}. In our proof we will follow the steps of \cite{IwaniecKovalevOnninen:W12approximation} and \cite{IwaniecKovalevOnninen:W1papproximation}. Namely, using Theorem \ref{Intro:MonotoneSobolev:Theorem}, we will first partition $\Omega$ into disjoint open subsets $\Upsilon_i$ and smooth the function $u$ in these subsets by replacing it with a $p$-harmonic function. This is the first step below. Next, we have to mollify the new function along the boundaries of $\Upsilon_i$. For this purpose, we partition even further $\Omega$, so that the boundaries of the regions $\Upsilon_i$  are contained in regions $\Psi_j$. Following \cite{IwaniecKovalevOnninen:W12approximation}, we call this new partition a \textit{lens-type partition}. In the regions $\Psi_j$ we consider another $p$-harmonic replacement. This completes the second step. In the third and final step we have to smooth our function along the boundaries of the regions $\Psi_j$, using smoothing results from Section \ref{Section:Smoothing}. Throughout all steps we have to ensure that our functions are monotone. This is guaranteed by the gluing lemmas from Subsection \ref{Section:glue} that allow us to paste together the various $p$-harmonic functions.

\subsection{Step 1: Approximation by a piecewise smooth function}

\subsubsection{Initial partition of $\Omega$}\label{InitialPartition:Section}
For fixed $\delta>0$ we consider a bi-infinite sequence of real numbers  $\{t_i\}_{i\in \Z}$ such that $t_i<t_{i+1}$, $|t_{i+1}-t_i|<\delta$ for $i\in \Z$, and $\lim_{i\to \pm\infty} t_i=\pm\infty$. Moreover, we may have that the conclusions of Theorem \ref{Sard:Theorem} are satisfied for the level sets $A_{t_i}$; that is, $A_{t_i}$ is an embedded $1$-dimensional topological submanifold of the plane. Note that $u^{-1}(t_i)$ or  $u^{-1}((t_i,t_{i+1}))$ will be empty if $t_i\notin u(\Omega)$ or $(t_i,t_{i+1})\cap u(\Omega)=\emptyset$, respectively.

\subsubsection{Solution of the Dirichlet Problem}\label{Dirichlet:Section}
In each region $ \Upsilon_i =u^{-1}((t_i,t_{i+1}))$, using Proposition \ref{DirichletProblem:Prop}, we solve the Dirichlet Problem with boundary data $u$ and obtain a function $u_i\in W^{1,p}(\Upsilon_i)$. The function $u_i$ is monotone in $\Upsilon_i$ since it is $p$-harmonic. By Lemma \ref{Comparison:Lemma}, we have $t_i\leq u_i\leq t_{i+1}$.

Moreover, if $x_0\in \partial \Upsilon_j \cap \Omega$, then $x_0 \in A_t$, where $t=t_j$ or $t=t_{j+1}$ by Lemma \ref{EssentialArcs:Lemma}(ii). Part (iii) of the same lemma implies that there exist disjoint Jordan regions $V_+, V_-$ such that the closure of $V_+\cup V_-$ contains a neighborhood of $x_0$ and there exists an open Jordan arc $J\subset \partial V_+\cap \partial V_-\cap A_t$ containing $x_0$ such that $J$ is an essential boundary arc of $ V_+$ and $ V_-$. Moreover, $V_+ \subset \Upsilon_j$ and $V_- \subset \Upsilon_i$ for some $i\in \Z$. By Lemma \ref{ContinuousExtension:Lemma} the functions $u_j,u_i$ extend continuously to $V_+\cup J$, $V_-\cup J$, respectively, and $u_j\big |J=t$, $u_i\big|J=t$. Since $\br{V_+\cup V_-}$ contains a neighborhood of $x$, we have that $u_j$ extends continuously to $x_0$. Hence, $u_j$ extends continuously to each point of $\partial \Upsilon_j\cap \Omega$ and agrees there with $u$.

Using Corollary \ref{Gluing:Corollary}, we ``glue" the functions $u_i$, $i\in \Z$, together with $u$, to obtain a continuous monotone function $u_\delta$ on $\Omega$. Note that 
$$u_\delta-u= \sum_{i\in \Z} [u_i-u]_0,$$
where $[u_i-u]_0 \in W^{1,p}_0(\Omega)$ denotes the extension of $u_i-u$ by $0$ outside $u^{-1}((t_i,t_{i+1}))$. The completeness of $W^{1,p}_0(\Omega)$ implies that $u_\delta-u\in W^{1,p}_0(\Omega)$. Therefore, by Proposition \ref{DirichletProblem:Prop} we conclude that 
$$E_p(u_\delta)\leq E_p(u)= \int_\Omega |\nabla u|^p.$$
If $u$ is not already $p$-harmonic, which we may assume, then the above inequality is strict by the uniqueness part of Proposition \ref{DirichletProblem:Prop}.

Since, $t_i\leq u_\delta\leq t_{i+1}$ in $ A_{t_i,t_{i+1}}$, we have
$$|u-u_\delta| \leq |u-t_i|+|t_i-u_\delta|\leq 2|t_{i+1}-t_i|<2\delta$$ 
in $\Omega$ and so $u_\delta$ is uniformly close to $u$. We also observe here that
\begin{align}\label{Dirichlet:vu}
u_\delta^{-1}((t_j,t_{j+1}))\subset A_{t_j,t_{j+1}}, \,\,\, j\in \Z,
\end{align}
which will be used later. This follows from the decomposition of $\Omega$ as
$$\Omega= \bigcup_{i\in \Z} (A_{t_i,t_{i+1}} \cup A_{t_i})$$
and the observation that $u_\delta^{-1}((t_i,t_{i+1}))$ cannot intersect $A_{t_j}$ for any $i,j\in \Z$, since $u_\delta=u=t_j$ on $A_{t_j}$.

Since $\Omega$ is bounded, it follows that $u_\delta \to u$ in $L^p(\Omega)$ as $\delta\to 0$. Moreover, $\|u_\delta\|_{W^{1,p}(\Omega)}= \|u_\delta\|_{L^p(\Omega)}+ \|\nabla u_\delta\|_{L^p(\Omega)}$ is uniformly bounded as $\delta\to 0$. By the Banach-Alaoglu theorem it follows that, along a sequence of $\delta\to 0$, $u_\delta$ converges weakly in $W^{1,p}(\Omega)$ to $u$, which is also the pointwise limit of $u_\delta$. Since the limits are unique, we do not need to pass to a subsequence. Hence, by Fatou's lemma for weak convergence 
\begin{align*}
\|\nabla u\|_{L^p(\Omega)}\leq \liminf_{\delta\to 0} \|\nabla u_\delta\|_{L^p(\Omega)} \leq \limsup_{\delta\to 0} \|\nabla u_\delta\|_{L^p(\Omega)}.
\end{align*}
The latter is bounded by $E_p(u)^{1/p}=\|\nabla u\|_{L^p(\Omega)}$, hence $\lim_{\delta\to 0}\|\nabla u_\delta\|_{L^p(\Omega)} = \|\nabla u\|_{L^p(\Omega)}$. This implies that $\nabla u_\delta\to \nabla u$ in $L^p(\Omega)$, due to the uniform convexity of $L^p$ spaces, when $1<p<\infty$; see for example \cite[pp.~95--97]{Brezis:Functional} and \cite[Proposition 3.32, p.~78]{Brezis:Functional}. 

We remark that $u_\delta$ might be constant in some component of a level region $u^{-1}((t_i,t_{i+1}))$. Summarizing, for each $\varepsilon>0$ there exists $\delta_0>0$ such that for $0<\delta<\delta_0$ there exists a monotone function $u_\delta$ on $\Omega$ satisfying the following:
\begin{enumerate}[(A)]
\item $u_\delta$ is $p$-harmonic in $u^{-1}((t_i,t_{i+1}))$, $i\in \Z$,
\item $\sup_{\Omega}|u_\delta-u|<\varepsilon$,
\item $\|\nabla u_\delta -\nabla u\|_{L^p(\Omega)}<\varepsilon$,
\item $E_p(u_\delta)< E_p(u)$, and
\item $u_\delta-u \in W^{1,p}_0(\Omega)$.
\end{enumerate}

In the next step, our goal is to approximate in the sense of (B)--(E) the function $v\coloneqq u_\delta$, for a small fixed $\delta$, by functions that are smooth along the level sets $A_{t_i}$, $i\in \Z$; note that these are the level sets of the original function $u$!

\subsection{Step 2: Smoothing along the level sets $u^{-1}(t_j)$}

\subsubsection{Lens-type partition}\label{Lens:Section}

Note that if a level set $v^{-1}(t)$ is a $1$-dimensional manifold, then it cannot intersect any component of $\bigcup_{i\in \Z}A_{t_i,t_{i+1}}$ where $v$ is constant. Since $v$ is $p$-harmonic in  $\bigcup_{i\in \Z}A_{t_i,t_{i+1}}$, it has at most countably many critical points in each component of $A_{t_i,t_{i+1}}$, unless it is constant there; see Section \ref{Section:Preliminaries}.  For $\eta>0$ and for each $j\in \Z$ we consider real numbers $t_j^-, t_j^+$ with $|t_j^+-t_j^-|<\eta$ such that $t_{j-1}^+<t_j^- <t_j<t_j^+<t_{j+1}^-$, $v^{-1}(t_j^{\pm})$ does not contain any critical points of $v$, and the conclusions of Theorem \ref{Sard:Theorem} hold for the level sets $v^{-1}(t_j^\pm)$. In particular, the level sets $v^{-1}(t_j^{\pm})$ intersect only components of $\bigcup_{i\in \Z}A_{t_i,t_{i+1}}$ in which $v$ is non-constant, and therefore the critical points of $v$ form a discrete set within these components. It follows that each point $x\in v^{-1}(t_j^\pm)$ has a neighborhood in  $\bigcup_{i\in \Z}A_{t_i,t_{i+1}}$ that does not contain any critical points of $v$.

In each region $\Psi_j\coloneqq v^{-1}( (t_{j}^-, t_j^+))$ we solve the Dirichlet Problem with boundary data $v$ as in Section \ref{Dirichlet:Section} and obtain $p$-harmonic functions $v_j$ with $t_{j}^-\leq v_j\leq t_j^+$ that extend continuously to $\partial \Psi_j\cap \Omega$ and agree there with $v$. We define a function $v_\eta$ to be equal to $v_j$ in $\Psi_j$, $j\in \Z$, and equal to $v$ on  $\Omega\setminus \bigcup_{j\in \Z} \Psi_j$. 

This procedure results in a continuous monotone function $v_\eta$, by Corollary \ref{Gluing:Corollary}, that  approaches $v$ in the uniform norm and also in $W^{1,p}(\Omega)$, as $\eta\to 0$. Moreover, we have  $E_p(v_\eta)< E_p(v)$ (unless $v$ is $p$-harmonic in $\Omega$, which we may assume that is not the case) and $v_\eta-v\in W^{1,p}_0(\Omega)$. That is, (B)--(E) from Section \ref{Dirichlet:Section} are satisfied.

\subsubsection{Regularity of the approximation}\label{Section:LipschitzRegularity}

As far as the regularity of $w\coloneqq v_\eta$ is concerned, we claim the following:
\begin{enumerate}[(a)]
\item The function $w$ is $p$-harmonic in $\Psi_j$ and in $v^{-1}((t_j^+,t_{j+1}^-))$, for all $j\in \Z$. Therefore, $w$ is $C^{1,\alpha}$-smooth in each of these sets and if $p=2$, it is actually $C^\infty$-smooth.
\item Each point $x_0\in v^{-1}(t_j^\pm)$, $j\in \Z$, has a neighborhood $V=V_+\cup V_-$ in 
$$\mathcal W\coloneqq \Omega \setminus \bigcup_{i\in \Z} v^{-1}(\{t_i^-,t_i^+\})= \bigcup_{i\in \Z} (\Psi_i \cup v^{-1}((t_i^+,t_{i+1}^-))),$$
where $V_+,V_-$ are disjoint Jordan regions contained in connected components of $\mathcal W$, $\br V$ contains a neighborhood of $x_0$, and there exists a $C^\infty$-smooth open Jordan arc $J\subset \partial V_+\cap \partial V_- \cap v^{-1}(t_j^{\pm})$ containing $x_0$ such that $J$ is an essential boundary arc of $V_+$ and $V_-$. Moreover, inside each one of $V_+, V_-$, $|\nabla w|$ is either vanishing or bounded below away from $0$; in the second case we have $w\neq t_{j}^\pm$ in the corresponding region. The function $w$ is smooth up to the boundary arc $J$ in each of $V_+,V_-$.
\item The function $w$ is $C^{\infty}$-smooth in $\mathcal W$, except possibly for a discrete set of points of $\mathcal W$ that accumulates at $\partial \Omega$.
\end{enumerate}

For (a) note that, by definition, $w$ is $p$-harmonic on $\Psi_j$, while on $v^{-1}((t_j^+ ,t_{j+1}^-)) \subset \Omega \setminus \bigcup_{i\in \Z} \Psi_i$ we have $w=v$ (by the definition of $w$) and $v$ is $p$-harmonic in $v^{-1}((t_j^+ ,t_{j+1}^-))=u_\delta^{-1}((t_j^+ ,t_{j+1}^-))\subset u^{-1}( (t_j,t_{j+1}))$; see \eqref{Dirichlet:vu}. Therefore, by the regularity of $p$-harmonic functions (see Section \ref{Section:Preliminaries}), $w$ is $C^{1,\alpha}$-smooth in the open set $\mathcal W$ and in fact $w$ is $C^\infty$-smooth in $\mathcal W$, except possibly for the (at most) countably many critical points that are contained in components $U$ of $\mathcal W$ in which $w$ is non-constant. The critical points form a discrete subset of such components $U$ of $\mathcal W$, so they could only accumulate in points of $\partial U$. 

Assuming that (b) is true, we will show (c). It suffices to show that the critical points contained in a component $U$ of $\mathcal W$ where $w$ is non-constant do not have any accumulation point in $\Omega$, but they can only accumulate at $\partial \Omega$. Suppose for the sake of contradiction that the critical points contained in $U$ accumulate at a point $x_0\in\partial U\cap \Omega$. Observe that the values $t_i^\pm$, $i\in \Z$, form a discrete set of real numbers with no finite accumulation points, and the level sets $v^{-1}(t_{i}^{\pm})$, $i\in \Z$, are embedded $1$-dimensional manifolds, as in the setting of Lemma \ref{EssentialArcs:Lemma}. By Lemma \ref{EssentialArcs:Lemma}(ii) there exists $j\in \Z$ such that $x_0 \in  v^{-1}(\{t_j^-, t_j^+\})$. Consider the Jordan regions $V_+,V_- \subset \mathcal W$ as in (b) such that $\br {V_+ \cup V_-}$ contains a neighborhood of $x_0$. This implies that there are infinitely many critical points of $w$ in, say, $V_+$, where $V_+\subset U$; note that at least one of $V_+,V_-$ has to be contained in $U$. Since $w$ is non-constant on $U\supset V_+$, the conclusion of (b) implies that $w$ does not have any critical points in $V_+$, a contradiction.

Finally we prove (b).   Let $x_0\in v^{-1}(t_j^-)$ for some $j\in \N$; the argument is the same if $x_0\in v^{-1}(t_j^+)$. By Lemma \ref{EssentialArcs:Lemma}(iii) we conclude that there exist disjoint Jordan regions $V_+,V_-$ such that the closure of $V_+\cup V_-$ contains a neighborhood of $x_0$ and there exists an open Jordan arc $J\subset \partial V_+\cup \partial V_- \cap v^{-1}(t_j^-) $ containing $x_0$ that is an essential boundary arc of $V_+$ and $V_-$. Moreover, the last statement of Lemma \ref{EssentialArcs:Lemma}(iii) implies that $V_+,V_-$ are contained in connected components of $ v^{-1}(( t_{j-1}^+,t_j^-))\cup v^{-1}((t_j^-,t_j^+))$. We will show that $|\nabla w|$ is vanishing or bounded below in each of $V_+,V_-$ (after possibly shrinking them) and that $w$ is smooth up to the arc $J$ in $V_+$ and $V_-$. We work with $V_+$.

Let $U$ be the component of $v^{-1}(( t_{j-1}^+,t_j^-))\cup v^{-1}((t_j^-,t_j^+))$ containing $V_+$. If $w$ is constant in $U$ then $w$ is obviously smooth up to the arc $J$ inside $V_+$ and we have nothing to show. We suppose that $w$ is non-constant in $U$.

Assume first that $V_+\subset U$ is contained in $v^{-1}((t_{j-1}^+,t_{j}^-))$. Then $w=v<t_j^-$ on $V_+$ by the definition of $w$. By the choice of $t_j^-$ in Section \ref{Lens:Section}, each point of $ v^{-1}(t_j^-)\supset J$ has a neighborhood in  $\bigcup_{i\in \Z}A_{t_i,t_{i+1}}$ that does not contain any critical points of $v$. In particular, each point of the arc $J$ has a neighborhood in which $v$ is non-constant, $p$-harmonic, and $|\nabla v|$ is non-zero. It follows that $v$ is smooth in a neighborhood of $J$; see Section \ref{Section:Preliminaries}. Therefore, after shrinking the Jordan region $V_+$ if necessary, $w$ is smooth in $V_+$ up to the arc $J$, and each point of $J$ has a neighborhood in $V_+$ in which $|\nabla w|=|\nabla v|$ is bounded away from $0$. 

Suppose now that $V_+\subset U$ is contained in $\Psi_j=v^{-1}((t_j^-,t_j^+))$. We have $J\subset v^{-1}(t_j^-)$, so $w=v=t_j^-$. Since $w$ is $p$-harmonic on $V_+\subset U$ and non-constant, we must have $w>t_j^-$ on $V_+$ by the strong maximum principle \cite[p.~111]{Heinonenetal:DegenerateElliptic}. Note that the arc $J$ is $C^\infty$-smooth, because $J\subset v^{-1}(t_j^-)$ and the $t_j^-$ is a regular (i.e., non-critical) value of $v$ by the choice of $t_{j}^{-}$. Now we are exactly in the setting of Corollary \ref{NonDegeneracy:Corollary}, which implies that each point of $J$ has a neighborhood in $V_+$ that does not contain any critical points of $w$ and $|\nabla w|$ is bounded away from $0$. If we shrink the Jordan region $V_+$, we may have that these hold in $V_+$. By Corollary \ref{NonDegeneracy:Corollary} it also follows that $w$ is smooth up to the boundary in the region $V_+$. The proof of (b) is completed.\qed

\vspace{0.5em}

We have managed to obtain a function $w$ that is smooth along the level sets $u^{-1}(t_j)$, but it is not necessarily smooth along $v^{-1}(t_j^{\pm})$. It has, however, some smoothness up to $v^{-1}(t_j^\pm)$, as described in (b). In the next subsection we prove that $w$ can be $C^\infty$-smoothed at arbitrarily small neighborhoods of the level sets $v^{-1}(t_{j}^{\pm})$ so as to complete the proof of Theorem \ref{Intro:Approximation:Theorem}. For this purpose we utilize smoothing results from Section \ref{Section:Smoothing}. 

\subsection{Step 3: Smoothing along the level sets $v^{-1}(t_j^{\pm})$}
We fix $j\in \Z$ and a component $J$ of $v^{-1}(t_j^-)$. Recall that $J$ is a smooth $1$-manifold, since $t_j^-$ was chosen to be a regular value of $v$, and $v$ is $C^\infty$-smooth in a neighborhood of $v^{-1}(t_j^-)$. There are two cases: either $J$ is homeomorphic to $\R$ or it is homeomorphic to $S^1$ (see the classification in Theorem \ref{Classification:theorem}).

By claim (b) from Subsection \ref{Section:LipschitzRegularity}, each $x_0\in J$ has a neighborhood $V=V_+\cup V_-$ in $\Omega\setminus v^{-1}(t_j^-)$,  where $V_+,V_-$ are disjoint Jordan regions, $\br V$ contains a neighborhood of $x_0$ and there exists an open arc $J_0\subset J$ containing $x_0$ that is contained in the common boundary of $V_+$ and $V_-$. Moreover, $w$ is $C^\infty$-smooth in ${ V_+}\cup J_0$ and in ${V_-}\cup J_0$, and $|\nabla w|$ is either vanishing or bounded below away from $0$ in each of $V_+, V_-$; in the second case we have $w\neq t_j^-$ in the corresponding set. 

Assume first that $J$ is homeomorphic to $\R$. By Theorem \ref{Classification:theorem}, there exists a neighborhood $U$ of $J$ in $\R^2$ and a diffeomorphism $\phi\colon \R^2\to U$ such that $\phi(\R)=J$. The previous paragraph implies that there exist connected neighborhoods $\widetilde B^+$ and $\widetilde B^-$ of $\R$ in $\{y>0\}$ and $\{y<0\}$, respectively, such that $w\circ \phi$ is smooth in $\widetilde B^+ \cup \R$ and in $\widetilde B^-\cup \R$, and each point of $\R$ has a neighborhood in $\widetilde B^+\cup\widetilde B^-$ in which $|\nabla (w\circ \phi)|$ is either vanishing or bounded away from $0$. Moreover, in each of $\widetilde B^+$ and $\widetilde B^-$ the function $w\circ \phi$ is either constant or we have $w\circ \phi > t_j^-$ or $w\circ \phi<t_j^-$ . 

Therefore, there exist disjoint connected open sets $B^+=\phi^{-1}(\widetilde B^+)$ and $B^-=\phi^{-1}(\widetilde B^-)$ so that $B=B^+\cup B^-\cup J$ is an open set containing $J$ and $w$ satisfies one of the following conditions, with the roles of $B^+$,$B^-$ possibly reversed:
\begin{enumerate}
\item[(i)] $w=t_j^-$ on $J$, $w>t_j^-$ on $B^+$, $w<t_j^-$ on $B^-$,
\item[(i$'$)] $w=t_j^-$ on $J$, $w>t_j^-$ on $B^+$, $w=t_j^-$ on $B^-$,
\item[(i$''$)] $w=t_j^-$ on $J$, $w>t_j^-$ on $B^+$, $w>t_j^-$ on $B^-$,
\item[(i$'''$)] $w=t_j^-$ on $J$, $w=t_j^-$ on $B^+$, $w=t_j^-$ on $B^-$.
\end{enumerate}
Note that (i$'''$) immediately implies that $w$ is smooth at each point of $J$. 

If (i) holds, then $|\nabla w|$ must be non-zero in $B^+\cup B^-$, hence each point of $J$ has a neighborhood in $B^+\cup B^-$ in which $|\nabla w|$ is bounded away from $0$. We are now exactly in the setting of the smoothing Lemma \ref{SmoothingArc:Lemma}. Thus, there exists a monotone function $\widetilde w$ in $\Omega$ that is smooth in $B$, agrees with $w$ outside an arbitrarily small neighborhood $A$ of $J$, and is arbitrarily close to $w$ in the uniform and in the Sobolev norm. In particular, since $E_p(w)<E_p(u)$ (cf.\ (D) from Subsection \ref{Dirichlet:Section}) we may have $E_p(\widetilde w) < E_p(u)$. Finally, we have $\widetilde w -w \in W^{1,p}_0(\Omega)$ from Lemma \ref{SmoothingArc:Lemma}(d) as claimed in (E) in the statement of Theorem \ref{Intro:Approximation:Theorem}. If (i$'$) or (i$''$) holds instead, then we are in the setting of Lemma \ref{SmoothingArc:Lemma2} and obtain the same conclusions. 

Now, if $J$ is homeomorphic to $S^1$ we can find connected open sets $B^+$, $B^-$ so that $J$ is their common boundary. In this case, we have the following alternatives:
\begin{enumerate}
\item[(i)] $w=t_j^-$ on $J$, $w>t_j^-$ on $B^+$, $w<t_j^-$ on $B^-$,
\item[(i$'$)] $w=t_j^-$ on $J$, $w>t_j^-$ on $B^+$, $w=t_j^-$ on $B^-$,
\item[(i$'''$)] $w=t_j^-$ on $J$, $w=t_j^-$ on $B^+$, $w=t_j^-$ on $B^-$.
\end{enumerate}
We note that the alternative (i$''$) of the previous case does not occur here, since it would violate the monotonicity of $w$. We can now apply Lemma \ref{SmoothingCurve:Lemma} to smooth the function $w$ near $J$.

We apply the above smoothing process countably many times in pairwise disjoint, arbitrarily small neighborhoods of components of $v^{-1}(t_j^{\pm})$, $j\in \Z$. Since these components do not accumulate in $\Omega$ (see Lemma \ref{EssentialArcs:Lemma}(i)), the resulting limiting function  $\widetilde w$ will be $C^\infty$-smooth in $\Omega$, except possibly at a discrete set of points of $\Omega$ (by (a),(c) in Subsection \ref{Section:LipschitzRegularity}), in which $\widetilde w$ is $C^{1,\alpha}$-smooth; if $p=2$ then $\widetilde w$ is $C^\infty$-smooth everywhere. If the neighborhoods of $v^{-1}(t_j^{\pm})$, $j\in \Z$, where the smoothing process takes place are arbitrarily small, then $\widetilde w$ will be $p$-harmonic on a subset of $\Omega$ having measure arbitrarily close to full.  Moreover, by Lemma \ref{ConvergenceMonotonicity:Lemma}, $\widetilde w$ is monotone in $\Omega$. Finally, it satisfies the corresponding conditions (B)--(E) from Subsection \ref{Dirichlet:Section}. This completes the proof of Theorem \ref{Intro:Approximation:Theorem}. \qed

\section{Smoothing results}\label{Section:Smoothing}

\begin{theorem}\label{Classification:theorem}
Let $J$ be an embedded $1$-dimensional smooth submanifold of $\R^2$ that is connected. Then $J$ is diffeomorphic to $X$, where $X=\R\times \{0\}\subset \R^2$ or $X=S^1\subset \R^2$. Moreover, there exists a neighborhood $U$ of $J$ in $\R^2$ and a diffeomorphism $\phi\colon \R^2\to U$ such that $\phi(X)=J$.
\end{theorem}

We direct the reader to \cite{Viro:1manifolds} for an account on the classification of $1$-manifolds. The last statement of the lemma can be proved in the same way as the existence of tubular neighborhoods of embedded submanifolds of $\R^n$; see for example \cite[Theorem 6.24, p.~139]{Lee:manifolds}.

\begin{lemma}[Smoothing along a line] \label{SmoothingArc:Lemma}
Let $\Omega\subset \R^2$ be an open set and let $J \subset \Omega$ be an embedded $1$-dimensional smooth submanifold of $\R^2$, homeomorphic to $\R$, such that $J$ has no accumulation points in $\Omega$; that is, the topological ends of $J$ lie in $\partial \Omega$. Suppose that $J$ is contained in an open set $B\subset \Omega$ that it is the common boundary of two disjoint regions $B^+$ and $B^-$ with $B=B^+\cup B^-\cup J$. Moreover, let $u\in W^{1,p}(\Omega)$ be a monotone function with the following properties:
\begin{enumerate}[\upshape(i)]
\item \label{SmoothingArc:Lemma:sign} $u=0$ on $J$, $u>0$ on $B^+$, $u<0$ on $B^-$,
\item \label{SmoothingArc:Lemma:smooth} $u$ is smooth in $B^+\cup J$ and in $B^-\cup J$,
\item \label{SmoothingArc:Lemma:nabla} each point of $J$ has a neighborhood in $B^+\cup B^-$ in which $|\nabla u|$ is bounded away from $0$.
\end{enumerate} 
Then for any open set $U\subset B$ with $U\supset J$ and each $\varepsilon>0$ there exists an open set $A\subset U$ containing $J$ and a monotone function $\widetilde u$ in $\Omega$ such that
\begin{enumerate}[\upshape(a)]
\item $\widetilde u$ agrees with $u$ in $\Omega\setminus A$ and is smooth in $B$,
\item $|\widetilde u- u|<\varepsilon$ in $A$, 
\item $\|\nabla \widetilde u-\nabla u\|_{L^p(A)}<\varepsilon$, and
\item $\widetilde u-u\in W^{1,p}_{0}(A)$.
\end{enumerate} 
\end{lemma}

\begin{proof}
By using a diffeomorphism $\phi$ defined in a neighborhood of $J$ in $B$, given by Theorem \ref{Classification:theorem}, we may straighten $J$ and assume that $J=\R=\{(x,y):y=0\}$, $B=\R^2$, $B^+=\{(x,y):y>0\}$, $B^-=\{(x,y):y<0\}$. For each $\varepsilon>0$ we will construct a smooth function $\widetilde u$ that agrees with $u$ outside an arbitrarily small neighborhood $A$ of $\R$, $|\widetilde u-u|<\varepsilon$ in $A$ and $\|\nabla \widetilde u -\nabla u\|_{L^q(A)}<\varepsilon$, where $q>p$ is a fixed exponent (e.g., $q=2p$). If the neighborhood $A$ is sufficiently thin, then $\widetilde u-u\in W^{1,q}_0(A)$. Pulling back $\widetilde u$ under the diffeomorphism $\phi$ and using H\"older's inequality will yield the desired conclusions in $\Omega$. We argue carefully for the monotonicity of $\widetilde u$ in the end of the proof. 

For the construction of $\widetilde u$, essentially, we are going to interpolate between the function $u$ away from $\R$ and the ``height function" $(x,y)\mapsto y$ near $\R$. Several technicalities arise though, in order to ensure that the new function is monotone.

By assumption \ref{SmoothingArc:Lemma:nabla}, for each $x\in \R$ there exists a constant $m>0$ such that $|\nabla u |>m$ in a neighborhood of $x$ in $\R^2\setminus \R$. Since $u=0$ on $\R$, we have $u_x=0$ on $\R$. Hence, using the smoothness of $u$ in $B^+\cup \R$ and in $B^-\cup \R$ we conclude that for each point $x\in \R$ there exist constants $m,M>0$ such that $M>|u_y|>m$ in a neighborhood of $x$ in $\R^2\setminus \R$. Thus, we may consider positive smooth functions $\gamma,\delta \colon \R \to \R$ such that $\delta(x)>|u_y(x,y)|>\gamma(x)>0$ for all $(x,y)$ lying in a small neighborhood of $\R$ in $\R^2\setminus \R$. By assumptions \ref{SmoothingArc:Lemma:sign} and \ref{SmoothingArc:Lemma:smooth}, we have $u_y\geq 0$ in a neighborhood of $\R$ in $\R^2\setminus \R$. Therefore, 
\begin{align}\label{SmoothingArc:Lemma:gammadelta}
\delta(x)>u_y(x,y)>\gamma(x)>0
\end{align} for all $(x,y)$ lying in a neighborhood of $\R$ in $\R^2\setminus \R$. By choosing a possibly larger $\delta$ we may also have
\begin{align}\label{SmoothingArc:Lemma:gammadelta2}
\delta(x)>|u_x(x,y)|
\end{align}
in that neighborhood. We let $A$ be a sufficiently small open neighborhood of $\R$ in $\R^2$ so that the preceding inequalities hold. Later we will shrink $A$ even further. 

For a positive smooth function $\beta\colon \R\to \R$ we define $V(\beta)=\{(x,y): |y|<\beta(x)\}$. Note that we may choose $\beta$ so that $\br{V(\beta)}\subset A$. By scaling $\beta$, we may also have that 
\begin{align}\label{SmoothingArc:Lemma:beta}
|\beta'|\leq 1.
\end{align}
Moreover, we consider a non-negative smooth function $\alpha\colon \R\to \R$ with $\alpha(t)=1$ for $t\geq 1$, $\alpha(t)=0$ for $t\leq 1/2$, and 
\begin{align}\label{SmoothingArc:Lemma:alpha}
0\leq \alpha'\leq 4.
\end{align}

We now define $s=s(x,y)= |y|/\beta(x)$ and 
\begin{align*}
\widetilde u(x,y)= \alpha(s)u(x,y)+ (1-\alpha(s))y\gamma(x).
\end{align*}
This function is smooth, agrees with $u$ outside $V(\beta)$ (where $s\geq 1$), and agrees with $(x,y)\mapsto y\gamma(x)$ in $V(\beta/2)$ (where $0\leq s<1/2$). We have
$|\widetilde u-u| \leq |1-\alpha(s)| |u-y\gamma|\leq |u|+ |y|\gamma$. By \eqref{SmoothingArc:Lemma:gammadelta} in $A$ we have 
\begin{align}\label{SmoothingArc:Lemma:u bound}
|u(x,y)| = \left| \int_0^y u_y(x,t)dt \right| \leq |y|\delta(x).
\end{align} 
Therefore, $|u|+|y|\gamma < |y|(\gamma+\delta)$.
If $A$ is so small that it is contained in the open set $\{(x,y): |y|<\varepsilon/(\gamma(x)+\delta(x))\}$, then we have $|\widetilde u-u|<\varepsilon$ in $A$, which proves claim (b).

We now compute the derivatives:
\begin{align*}
\widetilde u_x&= -\alpha'(s) \beta'(x) s^2 \frac{u-y\gamma}{|y|} +\alpha(s)u_x + (1-\alpha(s)) y\gamma' \,\,\,\,\, \textrm{and}\\
\widetilde u_y&= \alpha'(s)s \frac{u-y\gamma}{y}+ \alpha(s)u_y+(1-\alpha(s))\gamma.
\end{align*}
By \eqref{SmoothingArc:Lemma:beta}, \eqref{SmoothingArc:Lemma:alpha},  and \eqref{SmoothingArc:Lemma:u bound} we have 
\begin{align*}
|\widetilde u_x| &\leq 4(\gamma+\delta)+\delta +|y||\gamma'| \,\,\,\,\, \textrm{and}\\
|\widetilde u_y| & \leq 4(\gamma+\delta)+ \delta+ \gamma
\end{align*}
in $V(\beta)\subset A$. Since the above bounds and \eqref{SmoothingArc:Lemma:gammadelta}, \eqref{SmoothingArc:Lemma:gammadelta2} do not depend on $A$ but only on the function $u$, if $A$ is sufficiently small, then $\|\nabla \widetilde u- \nabla u\|_{L^q(A)}\leq  \|\nabla \widetilde u\|_{L^q(V(\beta))}+ \| \nabla u\|_{L^q(V(\beta))}$ can be made as small as we wish. 

We next prove that $\widetilde u$ is monotone in $\R^2$. Note that $\frac{u-y\gamma}{y} \geq 0$ on $A$, since $u_y>\gamma$. Therefore, 
\begin{align*}
\widetilde u_y \geq \gamma>0.
\end{align*}
This implies that $\nabla \widetilde u \neq 0$ in $A$, so $\widetilde u$ does not have any local maximum or minimum in $A$, and it is strictly monotone there. On the other hand, $\widetilde u =u$ outside $\br{V(\beta)}\subset A$, and therefore $\widetilde u$ is monotone outside $V(\beta)$. By the Gluing Lemma \ref{GluingMonotonicity:Lemma} it follows that $\widetilde u$ is monotone in $\R^2$. 

The argument of the previous paragraph and the use of Lemma \ref{GluingMonotonicity:Lemma} can be carried in $\Omega$, after precomposing $\widetilde u$ with the straightening diffeomorphism $\phi^{-1}$. We denote the composition, for simplicity, by $\widetilde u$. So $\widetilde u$ is a function in $\Omega$ that is strictly monotone in $\phi^{-1}(A)$ and can be extended to agree with $u$ in $\Omega\setminus \br{\phi^{-1}(V(\beta))}$; in particular, $\widetilde u$ is monotone in $\Omega\setminus \br{\phi^{-1}(V(\beta))}$. If we ensure that $\br{\phi^{-1}(V(\beta))}\cap \Omega \subset \phi^{-1}(A)$, then Lemma \ref{GluingMonotonicity:Lemma} can be applied to yield the monotonicity of $\widetilde u$ in $\Omega$. The latter can be achieved by shrinking $\beta$ (and thus the neighborhood $V(\beta)$) if necessary, using the assumption that the topological ends of $J$ lie in $\partial \Omega$ and not in $\Omega$. 
\end{proof}

\begin{lemma}\label{SmoothingArc:Lemma2}
The conclusions of Lemma \ref{SmoothingArc:Lemma} hold if the assumptions \textup{(i)},\textup{(iii)} are replaced by the assumptions
\begin{enumerate}\upshape
\item[(i$'$)] $u=0$ on $J$, $u>0$ on $B^+$, $u=0$ on $B^-$, and 
\item[(iii$'$)] each point of $J$ has a neighborhood in $B^+$ in which $|\nabla u|$ is bounded away from $0$,
\end{enumerate}
or if \textup{(i)} is replaced by the assumption
\begin{enumerate}\upshape
\item[(i$''$)] $u=0$ on $J$, $u>0$ on $B^+$, $u>0$ on $B^-$.
\end{enumerate}
\end{lemma}
\begin{proof}
In the first case, one argues as in the proof of Lemma \ref{SmoothingArc:Lemma}, by straightening the line $J$ to the real line with a diffeomorphism $\phi$. This time we interpolate between $u$ and the function $e^{-1/y}$ instead of the height function $y$:
\begin{align*}
\widetilde u(x,y)= \begin{cases} 
\alpha(s)u(x,y)+ (1-\alpha(s)) e^{-1/y} \gamma(x) & y>0\\
u & y\leq 0.
\end{cases}
\end{align*}
Here $\gamma$, $\delta$, $\beta$, $s$, $\alpha(s)$ are as in the previous proof, but working only in the upper half plane, and $A$, $V(\beta)=\{(x,y):0<y<\beta(x)\}$ are appropriate neighborhoods of $\R$ in the upper half plane $\{y>0\}$. The function $\widetilde u$ is smooth and the claims (a)--(d) from Lemma \ref{SmoothingArc:Lemma} follow as in the previous proof. We only have to argue differently for the monotonicity of $\widetilde u$. 

We compute for $y>0$
$$\widetilde u_y= \alpha'(s) s \frac{u-e^{-1/y}\gamma}{y}+ \alpha(s) u_y + (1-\alpha(s)) \frac{e^{-1/y}}{y^2}\gamma.$$
Since $u_y>\gamma$ on $A$, we have $\frac{u-e^{-1/y}\gamma}{y} \geq 0$ for all sufficiently small $y$. In particular this holds in $A$ if we shrink $A$. It follows that $\widetilde u_y>0$ on $A$. Thus, $\widetilde u$ is strictly monotone in $A$. Moreover, outside $\br {V(\beta)}\cap \{y\neq 0\} \subset A$ the function $\widetilde u$ agrees with $u$. Summarizing, in the domain $\widetilde\Omega =\R^2\setminus \R=\{y\neq 0\}$ we have $\br{V(\beta)} \cap \widetilde \Omega \subset A$ and the function $\widetilde u$ is strictly monotone in $A$, and monotone in $\widetilde \Omega\setminus \br{V(\beta)}$. By Lemma \ref{GluingMonotonicity:Lemma} we have that $\widetilde u$ is monotone in $\widetilde \Omega$.

Arguing as in the previous proof, we transfer the conclusions to $\Omega$ and deduce that $\widetilde u$ is monotone in $\Omega\setminus J$. Note that $J$ is a connected closed subset of $\Omega$. Since $\widetilde u=0$ in $J$ and $J$ exists all compact subsets of $\Omega$, by Lemma \ref{GluingConstant:Lemma} we conclude that $\widetilde u$ is monotone in $\Omega$. This completes the proof under the assumptions (i$'$) and (iii$'$).

Now, we assume that (i$'')$ holds in the place of (i). After straightening $J$, we define
\begin{align*}
\widetilde u(x,y)= \begin{cases} 
\alpha(s)u(x,y)+ (1-\alpha(s)) e^{-1/|y|} \gamma(x) & y\neq 0\\
0 & y= 0
\end{cases}
\end{align*}
This is a smooth function in $\R^2$. The conclusions (a)--(d) are again straightforward, so we only argue for the monotonicity. In $V(\beta)= \{(x,y): 0<|y|<\beta(x)\}$ and in a slightly larger open set $A\supset \br{V(\beta)}$ we have $\widetilde u_y\neq 0$. This implies that $\widetilde u$ is strictly monotone in $A$. Using Lemma \ref{GluingMonotonicity:Lemma}, we have that $\widetilde u$ is monotone in $\R^2\setminus \R$.

We transfer the conclusions to $\Omega$ and we obtain a function $\widetilde u$ that is monotone in $\Omega\setminus J$. Since $J$ is closed, connected and exits all compact subsets of $\Omega$, and $\widetilde u=0$ on $J$, by Lemma \ref{GluingConstant:Lemma} we conclude that $\widetilde u$ is monotone in $\Omega$. 
\end{proof}

\begin{lemma}[Smoothing along a Jordan curve] \label{SmoothingCurve:Lemma}
Let $\Omega\subset \R^2$ be an open set and let $J \subset \Omega$ be an embedded $1$-dimensional smooth submanifold of $\R^2$, homeomorphic to $S^1$. Suppose that $J$ is contained in an open set $B\subset \Omega$ that it is the common boundary of two disjoint regions $B^+$ and $B^-$ with $B=B^+\cup B^-\cup J$. Moreover, let $u\in W^{1,p}(\Omega)$ be a monotone function satisfying one of the following triples of conditions:
\begin{enumerate}[\upshape(i)]
\item  $u=0$ on $J$, $u>0$ on $B^+$, $u<0$ on $B^-$,
\item  $u$ is smooth in $B^+\cup J$ and in $B^-\cup J$,
\item  each point of $J$ has a neighborhood in $B^+\cup B^-$ in which $|\nabla u|$ is bounded away from $0$,
\end{enumerate} 
or
\begin{enumerate}[\upshape(i$'$)]
\item $u=0$ on $J$, $u>0$ on $B^+$, $u=0$ on $B^-$,
\item  $u$ is smooth in $B^+\cup J$ and in $B^-\cup J$,
\item  each point of $J$ has a neighborhood in $B^+$ in which $|\nabla u|$ is bounded away from $0$.
\end{enumerate} 
Then for any open set $U\subset B$ with $U\supset J$ and each $\varepsilon>0$ there exists an open set $A\subset U$ containing $J$ and a monotone function $\widetilde u$ in $\Omega$ such that
\begin{enumerate}[\upshape(a)]
\item $\widetilde u$ agrees with $u$ in $\Omega\setminus A$ and is smooth in $B$,
\item $|\widetilde u- u|<\varepsilon$ in $A$,
\item $\|\nabla \widetilde u-\nabla u\|_{L^p(A)}<\varepsilon$, and
\item $\widetilde u-u \in W^{1,p}_0(A)$.
\end{enumerate} 
\end{lemma}
\begin{proof}
We will only sketch the differences with the proofs of Lemmas \ref{SmoothingArc:Lemma} and \ref{SmoothingArc:Lemma2}. By Theorem \ref{Classification:theorem} we may map $B$ with a diffeomorphism to a neighborhood of $S^1$. After shrinking $B$, we may assume that this neighborhood of $S^1$ is an annulus. Then, using logarithmic coordinates, we map $S^1$ to $\R$. Let $\phi$ denote the composition of the two maps described. By, precomposing $u$ with $\phi^{-1}$, we can obtain a periodic function, still denoted by $u$, in a strip $\{(x,y):|y|< c\}$. We consider a strip $V(\beta)= \{ |y|<\beta\}$ under assumption (i), or $V(\beta)=\{0<y<\beta\}$ under (i$'$), where $\beta$ is a constant rather than a function,  and a strip $A \supset \br {V(\beta)}$. We consider the function $\widetilde u$ with the same definition as in the proofs of Lemmas \ref{SmoothingArc:Lemma} and \ref{SmoothingArc:Lemma2}. The function $\widetilde u$ is also periodic and smooth, so it gives a function in the original domain $\Omega$, by extending it to be equal to $u$ outside $B$. We still denote this function by $\widetilde u$. The properties (a)--(d) are straightforward to obtain, upon shrinking the strip $A$. We only argue for the monotonicity.

Under the first set of assumptions, and in particular under (i), the function $\widetilde u$ is strictly monotone in $\phi^{-1}(A)$ and agrees with $u$ outside $\phi^{-1}(\br{V(\beta)}) \subset A$. Hence, by Lemma \ref{GluingMonotonicity:Lemma}, $\widetilde u$ is monotone in $\Omega$.

Under the second set of assumptions, and in particular under (i$'$), we can only conclude that $\widetilde u$ is monotone in $\Omega\setminus J$; see also the proof of Lemma \ref{SmoothingArc:Lemma2}. However, now we cannot apply Lemma \ref{GluingConstant:Lemma}, since $J$ does not exit all compact subsets of $\Omega$. Instead, we are going to use Lemma \ref{Gluing:Lemma}. In $\phi^{-1}(V(\beta))$, which is precompact, by continuity we have $0<\widetilde u< t$ and $0<u<t$ for some $t>0$. We set $\Upsilon= u^{-1}((0,t))$, so we have $\phi^{-1}(V(\beta))\subset \Upsilon$. Observe that $\widetilde u$ is monotone in $\Upsilon \subset \Omega\setminus J$ and $0\leq \widetilde u\leq t$ in $\Upsilon$. By Lemma \ref{Gluing:Lemma} we conclude that the function
$$\widetilde u= \widetilde u \x_{\Upsilon} + u \x_{\Omega\setminus \Upsilon}$$
is monotone in $\Omega$. 
\end{proof}

\bibliography{biblio}

\begin{bibdiv}
\begin{biblist}

\bib{AlbertiEtAl:LipschitzLevelSets}{article}{
      author={Alberti, Giovanni},
      author={Bianchini, Stefano},
      author={Crippa, Gianluca},
       title={Structure of level sets and {S}ard-type properties of {L}ipschitz
  maps},
        date={2013},
     journal={Ann. Sc. Norm. Super. Pisa Cl. Sci. (5)},
      volume={12},
      number={4},
       pages={863\ndash 902},
}

\bib{AikawaKilpelainenShanmugalingamZhong:BoundaryHarnack}{article}{
      author={Aikawa, Hiroaki},
      author={Kilpel\"{a}inen, Tero},
      author={Shanmugalingam, Nageswari},
      author={Zhong, Xiao},
       title={Boundary {H}arnack principle for {$p$}-harmonic functions in
  smooth {E}uclidean domains},
        date={2007},
     journal={Potential Anal.},
      volume={26},
      number={3},
       pages={281\ndash 301},
}

\bib{FoundCompMath:book-Ball}{inproceedings}{
      author={Ball, J.M.},
       title={Singularities and computation of minimizers for variational
  problems},
        date={2001},
   booktitle={Foundations of computational mathematics},
      series={London Mathematical Society Lecture Note Series},
   publisher={Cambridge University Press},
     address={Cambridge},
       pages={1–\ndash 20},
}

\bib{Bates:Sard}{article}{
      author={Bates, S.~M.},
       title={Toward a precise smoothness hypothesis in {S}ard's theorem},
        date={1993},
     journal={Proceedings of the American Mathematical Society},
      volume={117},
      number={1},
       pages={279\ndash 283},
}

\bib{BojarskiHajlaszStrzelecki:Sard}{article}{
      author={Bojarski, B.},
      author={Haj{\l}asz, P.},
      author={Strzelecki, P.},
       title={Sard’s theorem for mappings in {H}{\"o}lder and {S}obolev
  spaces},
        date={2005},
     journal={Manuscripta Math.},
      volume={118},
       pages={383\ndash 397},
}

\bib{BourgainEtAl:Sard}{article}{
      author={Bourgain, Jean},
      author={Korobkov, Mikhail},
      author={Kristensen, Jan},
       title={On the {M}orse-{S}ard property and level sets of {S}obolev and
  {B}{V} functions},
        date={2013},
     journal={Rev. Mat. Iberoam.},
      volume={29},
       pages={1\ndash 23},
}

\bib{Brezis:Functional}{book}{
      author={Brezis, Haim},
       title={Functional analysis, {S}obolev spaces and partial differential
  equations},
      series={Universitext},
   publisher={Springer, New York},
        date={2011},
}

\bib{DePascale:Sard}{article}{
      author={de~Pascale, Luigi},
       title={The {M}orse-{S}ard theorem in {S}obolev spaces},
        date={2001},
     journal={Indiana University Mathematics Journal},
      volume={50},
      number={3},
       pages={1371\ndash 1386},
}

\bib{Dubovitskii:Sard}{article}{
      author={Dubovitskii, A.~Ya.},
       title={On the structure of level sets of differentiable mappings of an
  $n$-dimensional cube into a $k$-dimensional cube},
        date={1957},
     journal={Izv. Akad. Nauk SSSR Ser. Mat.},
      volume={21},
      number={3},
       pages={371\ndash 408},
}

\bib{EilenbergHarrold:Curves}{article}{
      author={Eilenberg, Samuel},
      author={Harrold, O.~G., Jr.},
       title={Continua of finite linear measure. {I}},
        date={1943},
        ISSN={0002-9327},
     journal={Amer. J. Math.},
      volume={65},
       pages={137\ndash 146},
}

\bib{Evans:C1alpha}{article}{
      author={Evans, Lawrence~C.},
       title={A new proof of local {$C^{1,\alpha }$} regularity for solutions
  of certain degenerate elliptic {P}.{D}.{E}.},
        date={1982},
     journal={J. Differential Equations},
      volume={45},
      number={3},
       pages={356\ndash 373},
}

\bib{Federer:gmt}{book}{
      author={Federer, Herbert},
       title={Geometric measure theory},
      series={Grundlehren der mathematischen Wissenschaften},
   publisher={Springer-Verlag New York Inc., New York},
        date={1969},
      volume={153},
}

\bib{Figalli:Sard}{article}{
      author={Figalli, Alessio},
       title={A simple proof of the {M}orse-{S}ard theorem in {S}obolev
  spaces},
        date={2008},
     journal={Proc. Amer. Math. Soc.},
      volume={136},
       pages={3657\ndash 3681},
}

\bib{GilbargTrudinger:pde}{book}{
      author={Gilbarg, David},
      author={Trudinger, Neil~S.},
       title={Elliptic partial differential equations of second order},
      series={Classics in Mathematics},
   publisher={Springer-Verlag, Berlin},
        date={2001},
        note={Reprint of the 1998 edition},
}

\bib{Heinonenetal:DegenerateElliptic}{book}{
      author={Heinonen, Juha},
      author={Kilpel\"{a}inen, Tero},
      author={Martio, Olli},
       title={Nonlinear potential theory of degenerate elliptic equations},
      series={Oxford Mathematical Monographs},
   publisher={The Clarendon Press, Oxford University Press, New York},
        date={1993},
}

\bib{HajlaszMaly:approximation}{article}{
      author={Haj{\l}asz, Piotr},
      author={Mal\'{y}, Jan},
       title={Approximation in {S}obolev spaces of nonlinear expressions
  involving the gradient},
        date={2002},
     journal={Ark. Mat.},
      volume={40},
      number={2},
       pages={245\ndash 274},
}

\bib{IwaniecKovalevOnninen:W1papproximation}{article}{
      author={Iwaniec, Tadeusz},
      author={Kovalev, Leonid~V.},
      author={Onninen, Jani},
       title={Diffeomorphic approximation of {S}obolev homeomorphisms},
        date={2011},
     journal={Arch. Ration. Mech. Anal.},
      volume={201},
      number={3},
       pages={1047\ndash 1067},
}

\bib{IwaniecKovalevOnninen:W12approximation}{article}{
      author={Iwaniec, Tadeusz},
      author={Kovalev, Leonid~V.},
      author={Onninen, Jani},
       title={Hopf differentials and smoothing {S}obolev homeomorphisms},
        date={2012},
     journal={Int. Math. Res. Not.},
      volume={2012},
      number={14},
       pages={3256\ndash 3277},
}

\bib{Lebesgue:Monotone}{article}{
      author={Lebesgue, H.},
       title={Sur le probl\`eme de dirichlet},
        date={1907},
     journal={Rend. Circ. Mat. Palermo},
      volume={27},
       pages={371\ndash 402},
}

\bib{Lee:manifolds}{book}{
      author={Lee, John~M.},
       title={Introduction to smooth manifolds},
     edition={Second edition},
      series={Graduate Texts in Mathematics},
   publisher={Springer, New York},
        date={2013},
      volume={218},
}

\bib{Lehrback:CapacityEstimate}{article}{
      author={Lehrb\"{a}ck, Juha},
       title={Pointwise {H}ardy inequalities and uniformly fat sets},
        date={2008},
     journal={Proc. Amer. Math. Soc.},
      volume={136},
      number={6},
       pages={2193\ndash 2200},
}

\bib{Lewis:C1alpha}{article}{
      author={Lewis, John~L.},
       title={Regularity of the derivatives of solutions to certain degenerate
  elliptic equations},
        date={1983},
     journal={Indiana Univ. Math. J.},
      volume={32},
      number={6},
       pages={849\ndash 858},
}

\bib{LewisNystrom:BoundaryRegularityPHarmonic}{article}{
      author={Lewis, John~L.},
      author={Nystr\"{o}m, Kaj},
       title={Regularity and free boundary regularity for the {$p$} {L}aplacian
  in {L}ipschitz and {$C^1$} domains},
        date={2008},
        ISSN={1239-629X},
     journal={Ann. Acad. Sci. Fenn. Math.},
      volume={33},
      number={2},
       pages={523\ndash 548},
}

\bib{Manfredi:pharmonicplane}{article}{
      author={Manfredi, Juan~J.},
       title={{$p$}-harmonic functions in the plane},
        date={1988},
     journal={Proc. Amer. Math. Soc.},
      volume={103},
      number={2},
       pages={473\ndash 479},
}

\bib{Manfredi:Monotone}{article}{
      author={Manfredi, J.J.},
       title={Weakly monotone functions},
        date={1994},
     journal={J. Geom. Anal.},
      volume={4},
      number={2},
       pages={393\ndash 402},
}

\bib{Moore:triods}{article}{
      author={Moore, R.~L.},
       title={Concerning triods in the plane and the junction points of plane
  continua},
        date={1928},
     journal={Proceedings of the National Academy of Sciences of the United
  States of America},
      volume={14},
      number={1},
       pages={85\ndash 88},
}

\bib{MalySwansonZiemer:Coarea}{article}{
      author={Mal\'{y}, Jan},
      author={Swanson, David},
      author={Ziemer, William~P.},
       title={The co-area formula for {S}obolev mappings},
        date={2003},
     journal={Trans. Amer. Math. Soc.},
      volume={355},
      number={2},
       pages={477\ndash 492},
}

\bib{Munkres:topology}{book}{
      author={Munkres, James~R.},
       title={Topology: a first course},
   publisher={Prentice-Hall, Inc., Englewood Cliffs, N.J.},
        date={1975},
}

\bib{Newman:Topology}{book}{
      author={Newman, M.H.A.},
       title={Elements of the topology of plane sets of points, second
  edition},
   publisher={Cambridge University Press, London},
        date={1951},
}

\bib{Norton:Sard}{article}{
      author={Norton, Alec},
       title={A critical set with nonnull image has large {H}ausdorff
  dimension},
        date={1986},
     journal={Trans. Amer. Math. Soc.},
      volume={296},
      number={1},
       pages={367\ndash 376},
}

\bib{Ntalampekos:CarpetsThesis}{misc}{
      author={Ntalampekos, Dimitrios},
       title={Potential theory on {S}ierpi{\'n}ski carpets with applications to
  uniformization},
   publisher={preprint arXiv:1805.12583},
        date={2018},
}

\bib{Pommerenke:conformal}{book}{
      author={Pommerenke, Ch.},
       title={Boundary behaviour of conformal maps},
      series={Grundlehren der Mathematischen Wissenschaften},
   publisher={Springer-Verlag, Berlin},
        date={1992},
      volume={299},
}

\bib{Rajala:uniformization}{article}{
      author={Rajala, Kai},
       title={Uniformization of two-dimensional metric surfaces},
        date={2017},
     journal={Invent. Math.},
      volume={207},
      number={3},
       pages={1301\ndash 1375},
}

\bib{RajalaRomney:reciprocal}{article}{
      author={Rajala, Kai},
      author={Romney, Matthew},
       title={Reciprocal lower bound on modulus of curve families in metric
  surfaces},
        date={2019},
     journal={Ann. Acad. Sci. Fenn. Math.},
      volume={44},
      number={2},
       pages={681\ndash 692},
}

\bib{Sard:Theorem}{article}{
      author={Sard, Arthur},
       title={The measure of the critical values of differentiable maps},
        date={1942},
     journal={Bull. Amer. Math. Soc.},
      volume={48},
      number={12},
       pages={883\ndash 890},
}

\bib{Semmes:Hausdorff}{misc}{
      author={Semmes, Stephen},
       title={{Some elementary aspects of {H}ausdorff measure and dimension}},
   publisher={Preprint arXiv:1008.2637},
        date={2010},
}

\bib{NoWildArcs}{article}{
      author={Thurston, W.},
       title={Why are there no wild arcs in the plane?},
        date={2011},
  url={https://mathoverflow.net/questions/57766/why-are-there-no-wild-arcs-in-the-plane},
        note={[Online; Accessed: 04-20-2019]},
}

\bib{Uralceva:C1alpha}{article}{
      author={Ural\cprime~ceva, N.~N.},
       title={Degenerate quasilinear elliptic systems},
        date={1968},
     journal={Zap. Nau\v{c}n. Sem. Leningrad. Otdel. Mat. Inst. Steklov.
  (LOMI)},
      volume={7},
       pages={184\ndash 222},
}

\bib{Viro:1manifolds}{article}{
      author={Viro, Oleg},
       title={1-manifolds},
        date={2013},
     journal={Bulletin of the Manifold Atlas},
}

\bib{Whitney:SardOptimal}{article}{
      author={Whitney, Hassler},
       title={A function not constant on a connected set of critical points},
        date={1935},
     journal={Duke Math. J.},
      volume={1},
      number={4},
       pages={514\ndash 517},
}

\bib{Whyburn:topology}{book}{
      author={Whyburn, Gordon~Thomas},
       title={Analytic topology},
      series={American Mathematical Society Colloquium Publications},
   publisher={American Mathematical Society, New York},
        date={1942},
      volume={28},
}

\bib{Whyburn:TopologicalAnalysis}{book}{
      author={Whyburn, Gordon~Thomas},
       title={Topological analysis},
      series={Princeton Mathematical Series},
   publisher={Princeton University Press, Princeton, N. J.},
        date={1958},
      volume={23},
}

\bib{Willard:topology}{book}{
      author={Willard, Stephen},
       title={General topology},
   publisher={Addison-Wesley Publishing Co., Reading, Mass.-London-Don Mills,
  Ont.},
        date={1970},
}

\end{biblist}
\end{bibdiv}
\end{document}